\documentclass[a4paper, 10pt, DIV10, headinclude=false, footinclude=false, final]{scrartcl}

\usepackage{amsthm, amsmath, amssymb, amsfonts}
\usepackage[utf8]{inputenc}
\usepackage[english]{babel}
\usepackage{url}
\usepackage{mathtools}

\usepackage{tikz}
\usepackage{tikz-3dplot}
\tdplotsetmaincoords{70}{110}
\usepackage{pgfplots}
\usepackage{pgfplotstable}
\usepackage{booktabs}

\numberwithin{figure}{section}
\numberwithin{table}{section}
\numberwithin{equation}{section}

\newenvironment{abstr}[1]{ \vspace{.05in}\footnotesize
	\parindent .2in
	{\upshape\bfseries #1. }\ignorespaces}{\par\vspace{.1in}}
\newenvironment{Abstract}{\begin{abstr}{Abstract}}{\end{abstr}}
\newenvironment{keywords}{\begin{abstr}{Key words}}{\end{abstr}}
\newenvironment{AMS}{\begin{abstr}{AMS subject classifications}}{\end{abstr}}

\newtheorem{theorem}{Theorem}[section]

\newtheorem{proposition}[theorem]{Proposition}

\theoremstyle{definition}

\newtheorem{remark}[theorem]{Remark}

\DeclareMathOperator{\supp}{supp}
\DeclareMathOperator{\diam}{diam}
\DeclareMathOperator{\dist}{dist}
\DeclareMathOperator{\Div}{div}
\renewcommand{\Re}{\operatorname{Re}}

\newcommand{\nz}{\mathbb{N}}       % natural numbers
\newcommand{\gz}{\mathbb{Z}}       % entire numbers
\newcommand{\rz}{\mathbb{R}}       % real numbers
\newcommand{\cz}{\mathbb{C}}       % complex numbers
\newcommand\CB{\mathcal{B}}

\newcommand\CQ{\mathcal{Q}}
\newcommand\CR{\mathcal{R}}
\newcommand\CS{\mathcal{S}}
\newcommand\CT{\mathcal{T}}

\newcommand{\UN}{\textup{N}}

\usepackage{changes}

\allowdisplaybreaks[4]

\begin{document}
	
	\title{Computational high frequency scattering from high contrast heterogeneous media%
	\thanks{The authors would like to acknowledge the kind hospitality of the Erwin Schr\"odinger International Institute for Mathematics and Physics (ESI), where part of this research was developed under the frame of the Thematic Programme {\itshape Numerical Analysis of Complex PDE Models in the Sciences}.}
	}
	\author{Daniel Peterseim\footnotemark[2] \and Barbara Verf\"urth\footnotemark[2]}
	\date{}
	\maketitle
	
	\renewcommand{\thefootnote}{\fnsymbol{footnote}}
	\footnotetext[2]{Institut für Mathematik, Universität Augsburg, Universitätsstr. 14, D-86159 Augsburg, \texttt{\{daniel.peterseim, barbara.verfuerth\}@math.uni-augsburg.de}}
	\renewcommand{\thefootnote}{\arabic{footnote}}
	
	\begin{Abstract}
	This article considers the computational (acoustic) wave propagation in strongly heterogeneous structures beyond the assumption of periodicity.
	A high contrast between the constituents of microstructured multiphase materials can lead to unusual wave scattering and absorption, which are interesting and relevant from a physical viewpoint, for instance, in the case of crystals with defects. We present a computational multiscale method in the spirit of the Localized Orthogonal Decomposition and provide its rigorous a priori error analysis for two-phase diffusion coefficients that vary between $1$ and very small values. Special attention is paid to the extreme regimes of high frequency, high contrast, and their previously unexplored coexistence.
	A series of numerical experiments confirms the theoretical results and demonstrates the ability of the multiscale approach to efficiently capture relevant physical phenomena.
	\end{Abstract}
	
	

	\section{Introduction}
	\label{sec:introduction}
	
	Wave propagation in heterogeneous high-contrast materials has received a growing interest in recent years because the combination of microstructures and high contrast can lead to unusual material properties such as opening band gaps, artificial magnetism, or surface plasmons, see, for instance, \cite{EP04negphC,JJWM08phC,PSS06cloaking,PE03lefthanded}.
	In periodic structures, homogenization techniques based on two-scale convergence \cite{All92twosc,LNW02twosc} or Bloch-wave techniques \cite{ABV16blochhom,AC98blochhom,CV97blochwave} are able to reduce the problem's complexity and to provide effective models for the overall macroscopic behavior of the wave.
	In practical applications, however, the periodic structure may easily be destroyed by local defects due to production errors.
	Perturbation theory for infinite periodic structures indicates that the combination of certain local defects and wave numbers may (suddenly) lead to localized waves, whereas in the perfectly periodic case the corresponding wave cannot propagate, see \cite{AS04waveguide,JJWM08phC}.
While this can be highly favorable, local defects may also destroy the desired tuned (unusual) properties of the periodic structures.
	However, the above described analytical examination of the phenomena is quite limited to the periodic case with (mild) perturbations, where each case has to be considered separately.
	Moreover, it is not clear how these asymptotic results can predict the particular behavior for a given (finite) structure.
	
	Therefore, numerical simulations of wave propagation and scattering in heterogeneous high-contrast materials are crucial to predict and also understand more situations and phenomena.
	Yet, a direct discretization of realistic problems using standard (finite element) methods easily exceeds today's computational resources because the mesh has to resolve all the material inhomogeneities.
	Additional challenges and restrictions on the mesh size are encountered in the high-frequency regime by the so-called pollution effect, see \cite{BS00pollutionhelmholtz}.
	Therefore, multiscale methods need to be employed to reliably simulate wave propagation in high-contrast heterogeneous materials.

	By now, the challenges for wave propagation problems on the one side and for high-contrast coefficients on the other side mostly have been studied separately.
	To reduce the pollution effect for the Helmholtz equation with constant or low-contrast coefficients, a number of approaches has been designed and analyzed, such as (hybridizable) discontinuous Galerkin methods \cite{CLX13HDGhelmholtz,GM11HDGhelmholtz}, the $hp$-version of the finite element method	\cite{EM12helmholtz,Melenk2013,MS10,MS11helmholtz}, (plane wave) Trefftz methods \cite{GHP09planewaveDGh,HMP16pwdg,PPR15pwvem}, or the Localized Orthogonal Decomposition \cite{BGP15hethelmholtzLOD,GP15scatteringPG,P17LODhelmholtz}.
	High-contrast coefficients mainly have been studied for elliptic problems using multiscale finite element methods \cite{CGH10msfemhighcontrast,EGH13gmsfem,OZ11gfem} and the Localized Orthogonal Decomposition \cite{HM17lodcontrast, PS16lodcontrast}, just to name a few.
	Finally, we mention that high-contrast Helmholtz problems have been studied based upon homogenization results in \cite{OV16hmmhelmholtz}.
	We shall emphasize that high contrast in the diffusion coefficient can be triggered by the degeneration of the upper or the lower bound of the elliptic operator. In the context of time-harmonic wave propagation, the observed effects are quite different in the two regimes, see, for instance, \cite{JJWM08phC}. Locally, the variation of the diffusion coefficient can be rephrased in terms of a variable wave number. Small diffusion coefficients correspond to large  wave numbers which amplifies the aforementioned pollution effect and leads to much more critical conditions on the minimal numerical resolution.
	That is why we focus in this manuscript on the specific setting where the diffusion coefficient takes a small value $\varepsilon^2$ on one part of the domain and $1$ on the complement.
	
	The paper presents a computational multiscale method in the spirit of the Localized Orthogonal Decomposition (LOD) and provides its rigorous analysis for high-contrast Helmholtz problems beyond periodicity and scale separation.
	We build upon previous ideas for high-frequency wave propagation in \cite{BGP15hethelmholtzLOD,GP15scatteringPG,P17LODhelmholtz} and high-contrast elliptic equations in \cite{BP16lodporous,HM17lodcontrast,PS16lodcontrast}.
	The combination of these techniques and approaches yields a (quasi-)localized method which does not require structures in the coefficients such as periodicity.
	Furthermore, a scale separation between the wave length and the size of individual ``valleys'' (where the coefficient is very small) is not required.
	The multiscale method can be viewed as a generalized finite element method with special multiscale basis functions.
	As standard finite element functions on a coarse mesh alone do not yield a faithful approximation space, problem-dependent multiscale functions are added. The latter are defined as solutions of local fine-scale problems.
	Our numerical analysis shows that the (coarse) mesh, and, hence, the dimension of the space, is coupled linearly to the effective wave number.
	Moreover, the size of the fine-scale problems to solve scales only logarithmically with this effective wave number.
	This generalizes the results of \cite{BGP15hethelmholtzLOD,GP15scatteringPG,P17LODhelmholtz} for the constant or low-contrast coefficient case.
	We carefully address the high contrast in our analysis which is crucial to obtain the mildest dependencies on the contrast (encoded in the quantity $\varepsilon$) as possible.
	We provide a priori error estimates in various norms, which can be motivated from the periodic (high contrast) case and are of the expected order.
	The error analysis contains the stability constant of the exact solution (see \eqref{eq:helmhstab} and \eqref{eq:helmhinfsup}). Unfortunately, only little is known in the general setting about bounds for this constant that are explicit in the wave number and the coefficients. Hence, the multiscale method is proved to be contrast- and wave number-robust only under the assumption that the solution depends moderately on these parameters.
	
	The paper is organized as follows. In Section \ref{sec:setting} we introduce our model problem and illustrate properties of the solution in the periodic case.
	We introduce the multiscale method with all necessary notation in Section \ref{sec:LOD} and analyze its well-posedness and a priori errors in Section \ref{sec:LODnumana}.
	The main proofs are detailed in Section \ref{sec:proofs}.
	Finally, numerical experiments in Section \ref{sec:numexp} confirm our predicted theoretical results and show the applicability of our method for physically relevant cases concerning band gap simulations and localized defects.

	\section{Helmholtz problem with high contrast}
	\label{sec:setting}
	In this section, we are concerned with the Helmholtz problem with high contrast in general. First, we formulate our model setting in Section \ref{subsec:problem} and also introduce the necessary notation.
	In Section \ref{subsec:hom}, we then explain in detail the considered set-up by motivating it from homogenization theory and Bloch-wave analysis. The discussion moreover sheds a light on what approximation results we can expect for our multiscale method and compares it to low-contrast Helmholtz problems.

	\subsection{Model problem}
	\label{subsec:problem}
	Let $\Omega\subset \rz^d$ be a bounded polyhedral Lipschitz domain.
	We consider the following model problem: Find $u:\Omega \to \cz$ such that
	\begin{equation*}\begin{aligned}
	-\Div(A\nabla u)-k^2u &=f&&\text{in }\Omega,\\
	\nabla u\cdot n -iku&=0&&\text{on }\partial \Omega.
	\end{aligned}
	\end{equation*}
	Here, we assume $f\in L^2(\Omega)$ and $k\geq k_0>0$.
	Inhomogeneous (Robin) boundary conditions on $\partial\Omega$ can be treated following the theory of \cite{HM14LODbdry}, which we omit for simplicity. Other types of boundary conditions \cite{GP15scatteringPG,P17LODhelmholtz} and also perfectly matched layers \cite{gallistl:hal-01887267} have been studied for homogeneous media and their incorporation in the present work is possible as well.
	The scalar diffusion coefficient $A$ is piece-wise constant with respect to a quadrilateral background mesh $\CT_\varepsilon$ with mesh size $O(\varepsilon)$ and $0<\varepsilon\ll 1$. On each quadrilateral, $A$ takes either the value $\varepsilon^2$ or $1$.
	We introduce the non-overlapping partition $\Omega=\Omega_1\cup \Omega_\varepsilon$, where $\Omega_1$ and $\Omega_\varepsilon$ are the parts of the domain where $A$ takes the value $1$ or $\varepsilon^2$, respectively.
	We assume that  $\Omega_\varepsilon\subset \subset \Omega$, so that $A$ does not appear in the Robin boundary condition.
	This set-up requires some comments. Although we fix the two values of $A$, the present setting still allows much freedom since the geometry of $\Omega_\varepsilon$ is relatively flexible. In particular, no periodicity is assumed.
	We emphasize that the scaling of $A$ is independent of the space dimension. The special choice of $1$ and $\varepsilon^2$ induces a high contrast between the two materials, which is important to obtain unusual macroscopic effects in (periodic) homogenization theory, see Section \ref{subsec:hom} below for a detailed discussion.
	However, we also underline that the results of the present contribution can directly be transferred to other scalings of $A$, see Remark \ref{rem:scaling}.
	We stick with the above choice of $\varepsilon^2$ in $\Omega_\varepsilon$ for simplicity and since this seems to be the physically (most) relevant set-up.
	As mentioned in the introduction, we stress that the present choice of high contrast deals with very small values of $A$ in parts of the domain (in contrast to, for instance, \cite{CGH10msfemhighcontrast}).
	This scaling in particular implies an almost degeneracy of the equation and is therefore harder than the case where $A$ may vary between $1$ and very large values, see Remark \ref{rem:scaling}.
    Finally, note that the assumption of a piece-wise constant $A$ can be relaxed if it is still possible to construct a suitable interpolation operator as described in Section \ref{subsec:mesh} for that diffusion coefficient.
	
	Throughout the whole article, we use standard notation on Sobolev spaces. All functions are complex-valued if not mentioned otherwise. The complex conjugate of $v$ is denoted by $\overline{v}$.
	We write $a\lesssim b$ in short for $a\leq Cb$ with a constant $C$ independent of $k$, $\varepsilon$, the mesh size $H$, and the oversampling parameter $m$.
	Similarly, $a\approx b$ stands for $a\leq Cb$ and $b\leq C^\prime a$ with constants $C, C^\prime$ independent of $k$, $\varepsilon$, $H$, and $m$.
	We then consider the (weak) problem: Find $u\in H^1(\Omega)$ such that
	\begin{equation}
	\label{eq:modelproblem}
	 \CB(u, v):=\int_{\Omega} A\nabla u\cdot \nabla \overline{v}- k^2u\overline{v}\, dx-ik\int_{\partial \Omega}u\overline{v}\, d\sigma=\int_\Omega f\overline{v}\, dx \qquad\forall v\in H^1(\Omega).
	\end{equation}
	Let $\|\cdot\|$ denote the standard $L^2(\Omega)$-norm, we then define the following weighted (semi) norms:
	\begin{align}
		\label{eq:L2Anorm}
		\|u\|_{0, A}&:=\|A^{1/2} u\|,\\
		\label{eq:H1Aseminorm}
		|u|_{1, A}&:=\|A^{1/2}\nabla u\|,\\
		\label{eq:H1Anorm}
		\|u\|_{1, A, k}&:=\bigl(\|A^{1/2}\nabla u\|^2+k^2\| u\|^2\bigr)^{1/2}.
	\end{align}
	If necessary, we will indicate by an additional subscript the domain where the norm is taken.
	Note that $\|\cdot\|_{1, A, k}$ represents the energy norm associated with the sesquilinear form $\CB$.
	If we assume that $\dist(\Omega_\varepsilon, \partial \Omega)$ is bounded from below by an $\varepsilon$-independent constant, we can deduce the following trace inequality (cf.\ \cite{BrennerScott})
	\[k^{1/2}\|v\|_{\partial \Omega}\lesssim\|v\|_{1, A, k}.\]
	Otherwise, if $\Omega_\varepsilon$ is arbitrarily closed to $\partial \Omega$, we would obtain an additional $\varepsilon$-weight in the trace inequality of the following form
	\[\|v\|_{\partial \Omega}\lesssim \varepsilon^{-1}\|v\|_{1, A, k}.\]
	Since the trace inequality is needed for the continuity of $\CB$ with respect to the energy norm, we assume that $\dist(\Omega_\varepsilon, \partial \Omega)$ is bounded from below by an $\varepsilon$-independent constant. Then, the continuity constant of $\CB$ is $\varepsilon$-independent, which simplifies the analysis of the multiscale method below, cf. also Remark \ref{rem:epsilon}.
	
	The unique continuation principle and Fredholm theory provide existence and uniqueness of a (weak) solution to \eqref{eq:modelproblem} in the two-dimensional case or under certain geometrical assumptions on $\Omega_\varepsilon$ in general dimensions, see \cite{GS18helmhexstabil,GPS18helmhstabil} and the discussions therein.
	Note that if $\Omega_\varepsilon$ consists of a finite collection of bounded subdomains as we consider it in the numerical experiments, the unique continuation principle holds in general dimensions, see \cite{GPS18helmhstabil}.
	The solution depends continuously on the right-hand side in the sense of the stability estimate 
	\begin{equation}\label{eq:helmhstab}
		\|u\|_{1, A, k}\lesssim C_{\operatorname{stab}}(k, \varepsilon)\|f\|_{L^2(\Omega)}.
		\end{equation}
		However, the dependencies of the constant $C_{\operatorname{stab}}(k, \varepsilon)$ on $\varepsilon$ and $k$ are not known explicitly in general, even for homogeneous media; see \cite{Mel95gFEM} for a positive and \cite{BCGLL11trapping} for a negative result. 
		Recently, various works have studied stability for the heterogeneous Helmholtz equation, see \cite{BGP15hethelmholtzLOD, GPS18helmhstabil, MS17helmhstabil, ST17helmh1d}. We emphasize that the (common) assumptions on the coefficients are not even close to the target setting of this paper, but the recent work \cite{LSW19helmholtztrapping} indicates that polynomial in $k$ stability results hold for almost all $k$. 
		The stability estimates readily implies an inf-sup condition of the form
		\begin{equation}
		\label{eq:helmhinfsup}
		\inf_{v\in H^1(\Omega)}\sup_{\psi\in H^1(\Omega)}\frac{\Re\CB(v, \psi)}{\|v\|_{1, A, k}\|\psi\|_{1, A, k}}\geq \gamma(\varepsilon, k)
		\end{equation}
		with $\gamma(\varepsilon, k)\approx (1+kC_{\operatorname{stab}}(\varepsilon, k))^{-1}$, see, e.g., \cite{Mel95gFEM}.
		
	\subsection{Photonic crystals and homogenization results}
	\label{subsec:hom}
	To motivate some of our results and connect and differentiate them to the existing literature, we consider the case of a bounded scatterer with periodic micro-structure.
	To be more precise, we consider $\Omega_\varepsilon=\bigcup_{j\in \mathcal{I}\subset \gz^2}\varepsilon(j+\Sigma)$, where $\Sigma$ is an open Lipschitz subset of the unit square and $\mathcal{I}$ is an index set such that $\Omega_\varepsilon$ lies completely in a compact subset of $\Omega$.
	The last condition is necessary to have a uniformly bounded distance between $\Omega_\varepsilon$ and $\partial \Omega$.
	
	\paragraph*{Two-scale homogenization}
	In the perfectly periodic case and under the assumption that $\varepsilon$ is much smaller than the wave length $\lambda\sim k^{-1}$, homogenization using two-scale convergence \cite{All92twosc, LNW02twosc} gives information about the limit problem for fixed but arbitrary $k$ and $\varepsilon\to 0$.
	The surprising result is that one does not only obtain an effective diffusion matrix but also an effective wave number of the form $k^2\mu$, where $\mu$ is defined by an additional problem on $\Sigma$.
	The sign of $\mu$ depends on the choice of $k$, which highly influences the behavior of the macroscopic solution: For positive $\mu$, standard wave propagation is experienced whereas for negative $\mu$ evanescent waves form that decay exponentially inside the scatterer with the periodic microstructure.
	The above considerations hold for compactly embedded inclusions $\Sigma$, see \cite{BF04homhelmholtz}, whereas a layered structure results in a degenerate diffusion matrix and induces so-called surface plasmons, see \cite{BS13plasmonwaves}.
	
	Moreover, due to the high contrast inclusions there is no weak $L^2$-convergence of the exact solution $u$ (for positive $\varepsilon$) to the macroscopic limit solution (named $u_0$), but instead an additional finescale-type contribution in the inclusion has to be added to $u_0$, see \cite{All92twosc}.
	In contrast to problems without high contrast as in \cite{CS14hmmhelmholtz}, we therefore cannot hope to find a good macroscopic $L^2$-approximation of $u$ -- this is only possible in $\Omega_1$.
	However, with an additional correction, we can expect to recover the exact solution in a good $L^2$-sense, cf. \cite{OV16hmmhelmholtz} for corresponding numerical results.
	This explains parts of the a priori error estimates of our multiscale method in Section \ref{subsec:erroranalysis}. Roughly speaking, we find a macroscopic approximation which is close to the exact solution in the $A$-weighted $L^2$ norm $\|\cdot \|_{0, A}$ (which, in particular, means a good approximation on $\Omega_1$.)
	Adding a suitable correction, we obtain an approximation which is good in the $\|\cdot\|_{1, A, k}$ sense. This in particular implies a good approximation in the standard $L^2$ norm.

	\paragraph*{Bloch wave homogenization and photonic crystals}
	If the wave length $\lambda\sim k^{-1}$ is of the same order as the periodicity length $\varepsilon$ and hence, also the diameter of one high-contrast inclusion $\varepsilon\Sigma$, other types of resonances occur which can no longer by predicted by two-scale homogenization theory, see \cite{DS17wavenum}.
	Instead, Bloch wave homogenization methods are employed which are based on the Floquet-Bloch theorem, see \cite{AC98blochhom,CV97blochwave}.
	To get an idea of the (optical) properties of the finite material with periodic microstructure (also called a photonic crystal), one can consider its infinite counterpart on the whole space $\rz^d$.
	If the wave number $k$ is an eigenvalue of the elliptic operator  $-\nabla\cdot (A\nabla\cdot)$ on the whole space, the corresponding waves can propagate inside the (infinite) crystal.
	Due to the periodicity, the spectrum can be computed as the union of all spectra on one periodicity cell with varying quasi-periodic boundary conditions, the so-called Bloch spectra, see \cite{CV97blochwave}.
	Band gaps correspond to ``forbidden'' wave numbers inducing evanescent waves as described above. 
	A high contrast between the material properties allows to open a considerably large gap in the spectrum.
	However, when the periodic structure is destroyed by (local) defects, perturbation theory indicates that eigenvalues inside the band gaps may appear. The corresponding eigenfunctions are often localized near the defect, see \cite{JJWM08phC} for a general discussion and \cite{AS04waveguide} for the specific example of a line defect.
	As mentioned, this considers the infinite crystal with the hope that a considerably large finite photonic crystal will have similar properties.
	The multiscale method presented in this contribution does not rely on two-scale homogenization results and therefore, can also cope with the regime of Bloch wave homogenization.
	Hence, we can directly study finite photonic crystals with possible local defects and their properties.

	\section{Multiscale method}
	\label{sec:LOD}
	In this section, we introduce the multiscale method based on the ideas of the Localized Orthogonal Decomposition (LOD) \cite{MP14LOD,HP13oversampl}.
	First in Section \ref{subsec:mesh}, we need some more notation on meshes and in particular a suitable interpolation operator which pays attention to the high contrast.
	After these preliminaries, we can define our method in detail in Section \ref{subsec:LODdef}. We close with some remarks concerning implementation in Section \ref{subsec:implremark}.

	\subsection{Meshes, finite element spaces, and interpolation operators}
	\label{subsec:mesh}
	We cover $\Omega$ with a regular mesh $\CT_H$ consisting either of simplices or of parallelograms/parallel\-epipeds.
	The mesh is assumed to be shape regular in the sense that the aspect ratio of the elements of $\CT_H$ is bounded uniformly from below.
	We introduce the mesh size $H=\max_{T\in \CT_H}\diam T$. 
	Although we will assume $H=O(\varepsilon)$ for most of this article, see \eqref{eq:resolcondH}, note that $\CT_H$ does not necessarily resolve the interfaces between $\Omega_1$ and $\Omega_\varepsilon$.
	Associated with an element $T\in\CT_H$ we define its neighborhood as \[\UN(T)=\bigcup_{K\in \CT_H, T\cap K\neq \emptyset} K.\]
	Thereby, for any $m\in\nz_0$, the $m$-layer patches are defined inductively via $\UN^{m+1}(T)=\UN(\UN^m(T))$ with $\UN^0(T):=T$.
	The shape regularity implies that there is a bound $C_{\operatorname{ol}, m}$ (depending only on $m$) of the number of the elements in the $m$-layer patch, i.e.,
	\begin{equation}
	\label{eq:Colm}
	\max_{T\in \CT_H}\operatorname{card}\{K\in \CT_H: K\subset \UN^m(T)\}\leq C_{\operatorname{ol}, m}.
	\end{equation}
	Throughout this article, we assume that $\CT_H$ is quasi-uniform, which implies that $C_{\operatorname{ol}, m}$ grows at most polynomially with $m$.
	We discretize the space $H^1_0(\Omega)$ with the lowest order Lagrange elements over $\CT_H$, and denote this space by $V_H$.
	This means that $V_H=H^1(\Omega)\cap \CS^1(\CT_H)$, where $\CS^1(\CT_H)$ denotes the space of globally continuous functions which are polynomials of partial degree $\leq 1$ (for quadrilateral elements) or polynomials of total degree $\leq 1$ (for simplicial elements).
	In the case of quadrilateral meshes, one can easily exploit a possible (periodic) structure or pattern in the coefficient $A$, see Section \ref{subsec:implremark}.
	
	Let $I_H: H^1(\Omega)\to V_H$ denote a bounded local linear projection operator, i.e., $I_H\circ I_H = I_H$, with the following stability and approximation properties for all $v\in H^1(\Omega)$ and all $w\in \ker I_H$
	\begin{align}
	\label{eq:IHstab}
	|I_H v|_{1, A, T}&\lesssim |v|_{1, A, \UN(T)},\\
	\label{eq:IHapprox}
	\|w\|_{0, A, T}&\lesssim H|w|_{1, A, \UN(T)},\\
	\label{eq:IHstabenergy}
	\|I_Hv\|_{1, A, k, T}&\lesssim \|v\|_{1, A, k, \UN(T)},
	\end{align}
	where we recall that the constant hidden in $\lesssim$ is independent of $H$ and $\varepsilon$.
	We emphasize that we assume stability and approximation properties of $I_H$ in $A$-weighted norms, which is the crucial difference between the low-contrast and the high-contrast case.
	The stability in the energy norm \eqref{eq:IHstabenergy} can be deduced from \eqref{eq:IHstab} and \eqref{eq:IHapprox} at least in the following two cases: The mesh satisfies the resolution condition \eqref{eq:resolcondH} (as assumed later on), which restricts the mesh size significantly and may be too pessimistic.
	The other possibility is that for each $T\in \CT_H$, the quotient $\frac{|T|}{|T\cap\Omega_1|}$ is uniformly bounded from above, which mainly implies that $H$ may not be too fine.
	Anyhow, the crucial assumptions are \eqref{eq:IHstab} and \eqref{eq:IHapprox}.
	They have been verified under certain geometric assumptions using special $A$-weighted interpolants in \cite{HM17lodcontrast,PS16lodcontrast}.
	A possible choice of \cite{PS16lodcontrast} that we also use in our numerical implementation builds upon $A$-weighted local $L^2$ projections in the following sense.
	Let $\mathcal{N}_{\operatorname{int}}$ denote the interior nodes of $\CT_H$ and $\lambda_z$ the hat function associated with the node $z$.
	We then define
    \begin{equation}\label{eq:IHweighted}
    I_H(v) =\sum_{z\in \mathcal{N}_{\operatorname{int}}}P_{A, \omega_z}(v)(z) \lambda_z
    \end{equation}
    with the local $A$-weighted $L^2$-projection $P_{A, \omega_z}$ on the patch $\omega_z:=\{K\in \CT_H: z\in K\}$.
    More precisely, for any $v\in H^1(\omega_z)$, $P_{A, \omega_z}(v)\in V_H|_{\omega_z}$ is given by
    \[\int_{\omega_z}A P_{A, \omega_z}(v)\, \psi_H\, dx =\int_{\omega_z}Av\, \psi_H\, dx\qquad \text{for all}\quad \psi_H\in V_H|_{\omega_z}.\]
    Assumptions \eqref{eq:IHstab} and \eqref{eq:IHapprox} hold for this interpolant if the coefficient $A$ is quasi-monotone.
    For instance, in the two-dimensional case, this condition is satisfied if -- roughly speaking -- $\Omega_\varepsilon$ consists of small inclusions of diameter $\varepsilon$ that do not touch each other.
    This setup is considered in our numerical experiments. For more details on quasi-monotonicity we refer to \cite{PS16lodcontrast} and the originial works in the context of domain decomposition \cite{Dryja1996,10.1007/978-3-642-11304-8_21}.
	In \cite{HM17lodcontrast}, a Scott-Zhang-type construction is employed, where the geometrical assumptions on $\Omega_\varepsilon$ are not so easy to describe.
	We emphasize, however, that the periodic structure underlying our numerical experiments in Section \ref{sec:numexp} is mentioned as an example in \cite{HM17lodcontrast} (with circle inclusion instead of square inclusions) and is also analyzed in \cite{BP16lodporous}.
	For this choice, our numerical experiments indicate that for a high contrast, the above described $A$-weighted interpolation operator $I_H$ (cf. \cite{PS16lodcontrast}) is favorable over its unweighted variant
	\begin{equation}\label{eq:IHunweighted}
	I_H^1(v):=\sum_{z\in \mathcal{N}_{\operatorname{int}}}P_{1, \omega_z}(v)(z) \lambda_z,
	\end{equation}
	which utilizes local projections w.r.t. the standard $L^2$-norm. Note that this unweighted construction is used for low-contrast problems, e.g. in \cite{Pet15LODreview}.
	Moreover, we stress that especially the construction in \cite{PS16lodcontrast} can be extended to coefficients which take more than two distinct values (see the discussion in \cite{PS16lodcontrast}).

	\subsection{Definition of the method}
	\label{subsec:LODdef}
	We formulate a (generalized) Petrov-Galerkin method to discretize \eqref{eq:modelproblem}.
	The method uses special multiscale trial and test functions, see \cite{P17LODhelmholtz}. The idea is to interpret the kernel of the interpolation operator $I_H$ as the space of functions with finescale heterogeneities and to incorporate parts of these information by special localized finescale problems.
	
	Let $W(\UN^m(T)):=\{w\in \ker I_H\;,\,w=0 \text{ in } \Omega\setminus \UN^m(T)\}$. Note that $\ker I_H$ is infinite-dimensional and hence $W(\UN^m(T))$ is non-trivial. We introduce the element-wise localized corrector approximation $\CQ_{T, m}:V_H\to W(\UN^m(T))$ via
	\begin{equation}
	\label{eq:correclocalproblem}
	\CB_{\UN^m(T)}(\CQ_{T, m} v_H, w)=-\CB_T(v_H, w)\qquad \forall w\in W(\UN^m(T)).
	\end{equation}
	Here, $\CB_\omega$ denotes the restriction of $\CB$ to a subset $\omega\subset \Omega$.
	Note that \eqref{eq:correclocalproblem} requires only the local computation of finescale problems if $m$ is small. The parameter $m$ is commonly called oversampling parameter and will be coupled to the mesh size $H$ later on.
	Using the element-wise correctors $\CQ_{T, m}$ we define the localized corrector approximation via
	\begin{equation}
	\label{eq:correclocal}
	\CQ_m=\sum_{T\in \CT_H}\CQ_{T, m}.
	\end{equation}
	The dual localized corrector is $\CQ^*_mv_H:=\overline{\CQ_m(\overline{v_H})}$.
	
	The localized method now uses the pair $((\operatorname{id}+\CQ_m)V_H, (\operatorname{id}+\CQ^*_m)V_H)$ as the new trial and test spaces in the Petrov-Galerkin formulation: Find $u_{H, m}\in V_H$ as solution to
	\begin{equation}
	\label{eq:LOD}
	\CB((\operatorname{id}+\CQ_m)u_{H, m}, (\operatorname{id}+\CQ^*_m)v_H)=(f, (\operatorname{id}+\CQ^*_m)v_H) \qquad \forall v_H\in V_H.
	\end{equation}
	Note that $u_{H,m}$ lies in the standard finite element space and, hence, is called a macroscopic approximation since it cannot contain finescale information of the exact solution $u$.
	This information is only recovered in
	\begin{equation}\label{eq:uLODm}
	u_{\operatorname{LOD}, m}:=(\operatorname{id}+\CQ_m)u_{H, m}.
	\end{equation}
	The well-posedness of the scheme and the error between the exact and the numerical solution will be analyzed in the next sections.
	
	\subsection{Practical aspects}
	\label{subsec:implremark}
	
	The method presented in the previous section is not yet ready to use since the corrector problems \eqref{eq:correclocalproblem} are still infinite-dimensional.
	Moreover, the number of these problems to solve increases with smaller mesh size $H$ so that one has to think about an efficient computation.
	Both issues are briefly addressed in this section.
	For general aspects of the implementation, we refer to \cite{EHMP16LODimpl}.

	\paragraph*{Fully discrete version of the multiscale method} To discretize the corrector problems \eqref{eq:correclocalproblem}, we introduce a finescale shape-regular quadrilateral or simplicial mesh $\CT_h$ of $\Omega$, which resolves all features and discontinuities of $A$.
	Moreover, we assume that the mesh is fine enough that a direct finite element simulation on $\CT_h$ would yield a faithful reference solution $u_h$ to $u$.
	We stress, however, that this reference solution is not needed and not computed in our method.
	Only the local corrector problems are solved on the finescale mesh using the space $W_h(\UN^m(T)):=\{w\in H^1(\Omega)\cap \CS^1(\CT_h)\;,\,w=0\text{ in }\Omega\setminus\UN^m(T)\text{ and }I_Hw=0\}$.
	The assumption that $\CT_h$ is sufficiently fine (as described above) is required for two reasons: (i) It guarantees that the multiscale method \eqref{eq:LOD} still possesses a unique solution because for a sufficiently fine mesh, $\CB$ fulfills an inf-sup condition over $H^1(\Omega)\cap \CS^1(\CT_h)$ (which is a central argument in the proof of Theorem~\ref{thm:LODwellposed} below). (ii) It implies that the error committed by this additional discretization stays negligible.
	Due to the latter fact, we only analyze (for simplicity) the semi-discrete method as formulated in Section \ref{subsec:LODdef} and refer to, e.g., \cite{GP15scatteringPG} for details on the minor technical changes in the proofs for the fully discrete method.
	
	\paragraph*{Efficient computation of the correctors}
	First, we note that all corrector problems are independent and hence, can be easily computed in parallel.
	What is even more important is that additional structure in the coarse mesh  $\CT_H$ and the coefficient can be exploited to reduce the number of corrector problems.
	For instance, in the periodic setting the corrector problems are translation invariant if we use a quadratic mesh with a mesh size which is an integer multiple of the periodicity length. Apart from a few problems at the boundary, only one interior corrector problem needs to be solved, cf. \cite{GP15scatteringPG} for the homogeneous Helmholtz equation.
	Similar arguments allow to reduce the number of corrector problems for a (periodic) photonic crystal with a few local defects.
	If one has already computed the correctors for the periodic photonic crystals and then introduces local defects (as perturbations of the periodic structure), one can employ local error indicators as in \cite{HM17lodsimilar} to automatically detect locations where the correctors need to be recomputed.
	Numerical experiments in \cite{HKM19lodperturb} have shown that thereby a large part of the old correctors can be reused for local defects, which also lowers the computational complexity.
	Here, however, we do not consider this approach in more detail since we assume the geometry of $\Omega_\varepsilon$ to be a priori given and not as a deviation from a previous, more structured, set-up.
	
	\paragraph*{Efficient computation of the LOD solution}
	The efficient computation of the LOD solution as well as CPU timings  have been addressed in full detail in \cite{EHMP16LODimpl} for diffusion and eigenvalue problems. The algorithms turn over to the Helmholtz case in a very natural way.
	Without any assumptions on the structure of the coefficient, we need to solve $O(H^{-d})$ cell problems of type \eqref{eq:correclocal}. If these are discretized on a mesh of size $h$ (see above), the number of degrees of freedom for each problem is $O(m^d(H/h)^d)$. Given their coercivity and since $H$ resolves both the wave number and the oscillations of $A$, they can be solved efficiently.
	As pointed out in the previous paragraph, the number of cell problems can be drastically reduced for structured coefficients.
	The linear system of the LOD associated with \eqref{eq:LOD} is of small dimension, namely of $\operatorname{dim}(V_H)$. Note that numerical experiments show that $m=2,3$ is sufficient and therefore the sparsity pattern of the matrix is only slightly enlarged by a factor $O(m^d)$ in comparison to standard FEM. The conditioning is comparable to standard FEM on the coarse mesh $\CT_H$.
	The most efficient iterative solution of such systems is an open question in numerical analysis and beyond the scope of this article.
	The present approach does neither simplify nor overcomplicate this issue of numerical linear algebra. However, the tailored discretization allows the reduction of the degrees of freedom and hence the size of the system to almost optimal in the sense of sampling theory, whereas standard FEM would require much higher costs even in the regime of constant coefficients.

\section{Error analysis}
\label{sec:LODnumana}
In this section, we analyze the multiscale method of Section \ref{subsec:LODdef} for the high-contrast Helmholtz problem.
We first show in Section \ref{subsec:correctors} that under a reasonable resolution condition $kH\varepsilon^{-1}\lesssim 1$,  the corrector problems are coercive and therefore well-posed. 
Then, we give the main results of this paper in Section \ref{subsec:erroranalysis}, namely the inf-sup condition of the method as well as several a priori error estimates.
These results are derived under the above resolution condition and the oversampling condition $m\gtrsim |\log(kH\varepsilon^{-1})|$ if the original model problem satisfies a polynomial $k$- and $\varepsilon$-dependence of the stability constant $C_{\operatorname{stab}}(k, \varepsilon)$.
The resolution condition $kH\varepsilon^{-1}\lesssim 1$ should be compared to the usual resolution condition encountered for the quasi-optimality in standard finite element methods, which presumably is $k^2h\varepsilon^{-2}\lesssim 1$ at best in non-smooth scenarios. Note that the condition for existence of a FEM solution can be relaxed to $h(k\varepsilon^{-1})^\alpha\lesssim 1$, but one will always have $\alpha>1$ for standard FEM.

\subsection{The corrector problems}
\label{subsec:correctors}
We first consider an idealized variant of \eqref{eq:LOD} with ``$m=\infty$'', i.e., with $ \UN^m(T)=\Omega$ in the corrector problems \eqref{eq:correclocalproblem}.
For notational convenience, we denote by $W$ the kernel of the interpolation operator $I_H$. 
Define the ideal correction operator $\CQ:H^1(\Omega)\to W$ via
\begin{equation}
\label{eq:correc}
\CB(\CQ v, w)=-\CB(v, w)\qquad \forall w\in W.
\end{equation}
Note that $\CQ=\sum_{T\in \CT_H}\CQ_T$, where $\CQ_T$ solves \eqref{eq:correc} with right-hand side $\CB_T$.
Since the sesquilinear form $\CB$ is indefinite, we first have to show the well-posedness of these corrector problems.

\begin{proposition}\label{prop:correccoercive}
If the mesh size $H$, the wave number $k$, and diffusion parameter $\varepsilon$ satisfy the condition
\begin{equation}
\label{eq:resolcondH}
kH\varepsilon^{-1}\lesssim 1,
\end{equation}
we have the following equivalence of norms on $W$
\[|w|_{1, A}\leq \|w\|_{1, A, k}\lesssim |w|_{1, A}.\]
Moreover, the sesquilinear form $\CB$ is elliptic over $W$, i.e., there exists $\alpha>0$ (independent of $H$, $\varepsilon$, and $k$) such that 
\[\Re\CB(w, w)\geq \alpha |w|_{1, A}^2\gtrsim \alpha\|w\|_{1, A, k}^2 \quad \forall w\in W.\]
\end{proposition}

\begin{proof}
Similar to \cite{P17LODhelmholtz}, the crucial observation is that \eqref{eq:IHapprox} implies
\begin{align}
k\|w\|\lesssim k\varepsilon^{-1}\|w\|_{0, A}&\lesssim k H\varepsilon^{-1}\|A^{1/2}\nabla w\|
\end{align}
for all $w\in W$.
Thus, the energy norm $\|\cdot\|_{1, A, k}$ is equivalent to the weighted $H^1$ semi norm $|\cdot |_{1, A}$ if \eqref{eq:resolcondH} is satisfied.
Moreover, the $L^2$ part of the sesquilinear form $\CB$ can be absorbed in the gradient part if \eqref{eq:resolcondH} is satisfied, leading to the ellipticity.
\end{proof}

The ellipticity of $\CB$ over $W$ implies the unique solvability of \eqref{eq:correc} due to the theory Lax-Milgram-Babu\v{s}ka.
The same argument can be applied to obtain the well-posedness of the localized corrector problems \eqref{eq:correclocal}.

\begin{remark}
The resolution condition \eqref{eq:resolcondH} is a natural consequence and generalization of the resolution condition $kH\lesssim 1$ for the homogeneous or low-contrast Helmholtz equation, cf. \cite{GP15scatteringPG, P17LODhelmholtz}. 
In fact, on $\Omega_\varepsilon$, $k\varepsilon^{-1}$ is the effective wave number which needs to be resolved.
There are certainly geometries of $\Omega_\varepsilon$ where this condition is necessary and sharp, but it is in general not clear how ``likely'' such situations are.
As discussed, at least conditions like $\varepsilon^{-2}k^2h\lesssim 1$ and $h\ll \varepsilon$ are expected for general non-smooth coefficients for standard finite element methods where possible advantages of high-order methods \cite{MS10, MS11helmholtz} cannot easily be exploited. (An exact statement or estimate of this condition is not available for this setting as the dependence of the stability constant on $\varepsilon$ and $k$ is not known.)
\end{remark}

The error between the idealized corrector and its localized (or truncated) approximation decays exponentially in the number of layers $m$ (also called oversampling parameter) in the following way.

\begin{theorem}
	\label{thm:correcerror}
	Let $\CQ$ be defined by \eqref{eq:correc} and $\CQ_m$ be defined by \eqref{eq:correclocal}.
	Let \eqref{eq:resolcondH} be satisfied.
	Then there exists a constant $0<\beta<1$, independent of $H$, $m$, $k$, and $\varepsilon$, such that for all $v_H\in V_H$ it holds that
	\begin{equation}
	\label{eq:correcerror}
	|(\CQ-\CQ_m)v_H|_{1, A}\lesssim 
	 C_{\operatorname{ol}, m}^{1/2}\beta^m|v_H|_{1, A}.
	\end{equation}
\end{theorem}
The proof is given in Section \ref{subsec:proofcorrec}.
Employing that both correctors map into the kernel space and using \eqref{eq:IHapprox}, we deduce furthermore that
\[\|(\CQ-\CQ_m)v_H\|_{0, A}\lesssim C_{\operatorname{ol}, m}^{1/2}H\beta^m\|v_H\|_{1, A}.	\]
We emphasize that neither the constant hidden in $\lesssim$ nor $\beta$ depend on $H$, $m$, $k$, or $\varepsilon$. In particular the independence of $\beta$ from $\varepsilon$ is the so-called contrast-independent localization. It can only be obtained with the $A$-weighted interpolation operator and is the major difference of this work from the previous ones on the low-contrast Helmholtz equation \cite{BGP15hethelmholtzLOD,GP15scatteringPG,P17LODhelmholtz}.

\subsection{Stability of the method and a priori error estimates}
\label{subsec:erroranalysis}
The previous section showed that the corrector problems have a unique (and stable) solution.
Before analyzing the error of the multiscale method \eqref{eq:LOD}, we need to show the stability (and well-posedness) of the method by proving an inf-sup condition.
Since the arguments utilize the corrector error estimate \eqref{eq:correcerror}, we first examine what happens if we replace $\CQ_m$ and $\CQ^*_m$ in \eqref{eq:LOD} by their idealized counterparts $\CQ$ from \eqref{eq:correc} and its adjoint $\CQ^*$.
Hence, $\CQ^*: H^1(\Omega)\to W$ is defined via
\[\CB(w, \CQ^*v)=-\CB(w, v)\qquad \forall w\in W\]
and we note that $\CQ^*v:=\overline{\CQ(\overline{v})}$.
By the definition of $\CQ^*$ and $W$, it holds that $W$ and $(\operatorname{id}+\CQ^*)V_H$ are orthogonal with respect to $\CB$.
More precisely, we have $\CB(w, (\operatorname{id}+\CQ^*)v_H)=0$ for all $w\in W$ and all $v_H\in V_H$. Note that the same holds true for $(\operatorname{id}+\CQ)V_H$ and $W$.
Next we observe that the interpolation $I_H u$ of the exact solution satisfies
\begin{align*}
\CB(I_H u, (\operatorname{id}+\CQ^*)v_H)&=\CB(u, (\operatorname{id}+\CQ^*)v_H)+\CB(\underbrace{I_H u-u}_{\in W}, (\operatorname{id}+\CQ^*)v_H)\\
&= (f, (\operatorname{id}+\CQ^*)v_H) .
\end{align*}
Using once more the $\CB$-orthogonality of $W$ and $(\operatorname{id}+\CQ^*)V_H$, this implies that $u_{\operatorname{LOD}}:=(\operatorname{id}+\CQ) I_H u$ solves
\[\CB(u_{\operatorname{LOD}}, (\operatorname{id}+\CQ^*)v_H)= (f, (\operatorname{id}+\CQ^*)v_H) \qquad \forall v_H\in V_H.\]

The well-posedness and the a priori estimates now follow from the properties of the above idealized method and the fact that the ideal and the localized correctors are exponentially close.

\begin{theorem}
\label{thm:LODwellposed}
Under the resolution condition \eqref{eq:resolcondH} and the oversampling condition \begin{equation}
\label{eq:oversamplcond}
m\gtrsim |\log(\gamma^{-1}(\varepsilon, k)C_{\operatorname{ol}, m}^{1/2}
)|/|\log(\beta)|
\end{equation}
$\CB$ satisfies the following inf-sup condition: there exists $\gamma_{\operatorname{LOD}}\approx\gamma(\varepsilon, k)>0$ such that
\begin{equation}
\label{eq:LODinfsup}
\inf_{v_H\in V_H}\sup_{\psi_H\in V_H}\frac{\Re\CB((\operatorname{id}+\CQ_m)v_H, (\operatorname{id}+\CQ^*_m)\psi_H)}{\|v_H\|_{1, A, k}\| \psi_H\|_{1, A, k}}\geq \gamma_{\operatorname{LOD}}.
\end{equation}
\end{theorem}
 The proof is similar to \cite{P17LODhelmholtz} and needs the contrast-independent exponential decay of the corrector error, for details see Section \ref{subsec:proofestimates}.
 The oversampling condition depends on the stability constant for the original problem so that the dependence of $m$ on $\varepsilon$ and $k$ is not analytically known.
 If we assume $\gamma(\varepsilon, k)\approx k^{-q}\varepsilon^p$ for non-negative numbers $p$ and $q$ independent of $k$ and $\varepsilon$, the oversampling condition roughly reads $m\gtrsim |\log(k\varepsilon^{-1})|$ (with a multiplicative factor $pq$ hidden in the $\lesssim$ notation).
This logarithmic dependence on the effective wave number $k\varepsilon^{-1}$ in $\Omega_\varepsilon$ is in agreement with the results for the homogeneous or low-contrast Helmholtz equation in \cite{BGP15hethelmholtzLOD, GP15scatteringPG, P17LODhelmholtz}.
 Numerical experiments moreover indicate that rather small numbers $m=2,3$ are sufficient.

	\begin{theorem}
	\label{thm:LODapriori}
	Let $u$ be the solution to \eqref{eq:modelproblem} and let $u_{H, m}$ be the solution to \eqref{eq:LOD}  and set $u_{\operatorname{LOD}, m}:=(\operatorname{id}+\CQ_m)u_{H, m}$.
	Assume that \eqref{eq:resolcondH} and the oversampling condition \eqref{eq:oversamplcond} are satisfied.
	Then it holds that
	\begin{align}
	\label{eq:errorfineenergy}
	\|u-u_{\operatorname{LOD}, m}\|_{1, A, k}&\lesssim H\|A^{-1/2} f\|+C_{\operatorname{ol}, m}^{1/2}\beta^m\gamma^{-1}(\varepsilon, k)\|f\|,\\
	\label{eq:errorfineL2w}
	\|u-u_{\operatorname{LOD}, m}\|_{0, A}&\lesssim (H+C_{\operatorname{ol}, m}^{1/2}\beta^m\varepsilon\gamma^{-1}(\varepsilon, k)) \|u-u_{\operatorname{LOD}, m}\|_{1, A, k}\\
	\label{eq:errorcoarseL2w}
	\|u-u_{H, m}\|_{0, A}&\lesssim H\inf_{v_H\in V_H}|u-v_H|_{1, A}\\\nonumber
	&\quad+C_{\operatorname{ol}, m}^{1/2}\beta^m\gamma_{\operatorname{LOD}}^{-1}\bigl(H\|A^{-1/2}f\|+C_{\operatorname{ol}, m}^{1/2}\beta^m\gamma^{-1}(\varepsilon, k)\|f\|\bigr).
	\end{align}
	\end{theorem}
The proof is detailed in Section \ref{subsec:proofestimates}.

	If $m$ does not only fulfill the oversampling condition \eqref{eq:oversamplcond}, but is also coupled to $H$ like $m\approx |\log(\gamma^{-1}(\varepsilon, k)C_{\operatorname{ol}, m}^{1/2}H)|$, Theorem \ref{thm:LODapriori} simplifies to the following convergence orders: (i) The error between the exact and the full multiscale solution converges linearly in the energy norm and quadratically in an $A$-weighted $L^2$ norm, cf. \eqref{eq:errorfineenergy} and \eqref{eq:errorfineL2w}. (ii) The error between the exact solution and the FE part of the multiscale solution converges (at least) linearly in the $A$-weighted $L^2$ norm, cf. \eqref{eq:errorcoarseL2w}.
For the latter, note that  $\inf_{v_H\in V_H}|u-v_H|_{1, A}\lesssim |u|_{1, A}\lesssim C_{\operatorname{stab}}(\varepsilon, k)\|f\|$ independent of the regularity of $u$.
If $u$ has more than $H^1(\Omega)$ regularity, the rate in \eqref{eq:errorcoarseL2w} may improve.
We stress that the remaining terms on the right-hand side of \eqref{eq:errorcoarseL2w} are of order $H^2$ if the coupling $m\approx |\log(\gamma^{-1}(\varepsilon, k)C_{\operatorname{ol}, m}^{1/2}H)|$ is satisfied.

	These estimates generalize the (expected) approximation results formulated in Section \ref{subsec:hom}.
	We can only find a macroscopic approximation of the exact solution in $\Omega_1$, which is reflected in \eqref{eq:errorcoarseL2w} by the $A$-weighting in the $L^2$ norm. For a good approximation in the energy norm, which also includes the standard $L^2$ norm, we need an additional (finescale) corrector, which is also present in \eqref{eq:errorfineenergy}.
	In the main error estimate \eqref{eq:errorfineenergy}, we furthermore emphasize that  the error from the volume term is weighted by $A^{-1/2}$, which induces a factor of $\varepsilon^{-1}$ on the right-hand side of the estimate \emph{if} the support of $f$ intersects with $\Omega_\varepsilon$.
	
	\begin{remark}
	\label{rem:epsilon}
	Our main results of this paper rely on the $\varepsilon$-independent trace inequality mentioned in Section \ref{subsec:problem}, which needs that $\Omega_\varepsilon$ is bounded uniformly away from $\partial \Omega$.
	For settings where this is not the case, one can carefully trace the occurrence of extra $\varepsilon$ powers because of the continuity of $\CB$ and we find that this mainly changes the oversampling condition to $m\gtrsim|\log(\gamma^{-1}(\varepsilon, k)\varepsilon C_{\operatorname{ol}, m}^{1/2})|/|\log(\beta)|$.
	This, however, is not critical since we expect a dependence of $m$ on $\varepsilon$ anyway because of $\gamma(\varepsilon, k)$.
	\end{remark}	

	\begin{remark}
	\label{rem:scaling}
	The proofs of Theorems \ref{thm:LODwellposed} and \ref{thm:LODapriori} reveal that the $\varepsilon$ dependence in the resolution and the oversampling condition come from the $\varepsilon$ dependence of the lower bound on $A$ (and not the minimal diameter $\varepsilon$ of an individual ``valley'' of $\Omega_\varepsilon$).
	In fact, one can directly generalize the results to the case where $A$ jumps between the values $1$ and $\alpha^2\ll 1$ by replacing all occurrences of $\varepsilon$ by $\alpha$.
	In particular, this means that our results remain valid for other scalings of $A$ beyond the (physically interesting) one of, for instance, \cite{BF04homhelmholtz}.
	
	In the ``complementary'' high contrast setting, where $A$ jumps between values $1$ and $\alpha\gg 1$, one also needs $A$-weighted interpolation operators to get a contrast-independent decay of the correctors and thereby contrast-independent error estimates for the LOD.
	Since in this case it trivially holds $\|v\|\lesssim \|v\|_{0,A}$ for any $v\in H^1(\Omega)$, the resolution condition \eqref{eq:resolcondH} reduces to the usual $kH\lesssim 1$ (independent of $\alpha$), cf. for instance the proof of Proposition \ref{prop:correccoercive}.
	For the same reason, one can replace $\|A^{-1/2} f\|$ by $\|f\|$ in \eqref{eq:errorfineenergy} and \eqref{eq:errorcoarseL2w}.
    These improvements of the results reflect and underline the fact that this scaling of $A$ is ``easier'' as it does not lead to high (effective) wave numbers as discussed in the introduction.
	\end{remark}

	\section{Proofs of the main results}
	\label{sec:proofs}
	
	\subsection{Proof of the corrector error}
	\label{subsec:proofcorrec}
	The goal of this section is to prove \eqref{eq:correcerror}.
	The main ingredient is the following exponential decay result.
	We empasize once more that the following decay is contrast-independent, i.e., the rate $\tilde{\beta}$ and also the multiplicative constant hidden in the notation $\lesssim$ both do \emph{not} depend on $\varepsilon$. This is achieved by using an $A$-weighted interpolation operator and distinguishes the present analysis from previous works on the Helmholtz equation \cite{BGP15hethelmholtzLOD,GP15scatteringPG,P17LODhelmholtz}.
	
	\begin{proposition}
		\label{prop:correcdecay}
		Let $\CQ=\sum_{T\in \CT_H}\CQ_T$ be defined by \eqref{eq:correc} and $v_H\in V_H$. 
		There exists a constant $0<\tilde{\beta}<1$, independent of $H$, $m$, and $\varepsilon$ such that
		\begin{equation}
		\label{eq:correcdecay}
		|\CQ_T v_H|_{1, A, \Omega\setminus \UN^m(T)}\lesssim \tilde{\beta}^m|v_H|_{1, A, T}.
		\end{equation} 
	\end{proposition}
	
	\begin{proof}
		We define the cutoff function $\eta\in V_H$ via
		\[\eta=0\quad\text{in}\quad\UN^{m-3}(T),\qquad\eta=1\quad \text{in}\quad \Omega\setminus \UN^{m-2}(T)\]
		and set $\CR=\supp(\nabla \eta)$.
		Let $v_H\in V_H$ and denote $\phi:=\CQ_Tv_H$.
		Elementary estimates lead to
		\begin{align*}
		|\phi|^2_{1, A, \Omega\setminus\UN^m(T)}&\lesssim |\Re(A\nabla \phi, \eta\nabla\phi)_{\Omega}|\leq|\Re(A\nabla \phi, \nabla(\eta\phi))_{\Omega}|+|\Re(A\nabla \phi, (\nabla \eta)\phi)_{\Omega}|\\
		&\leq M_1+M_2+M_3
		\end{align*}
		with
		\begin{align*}
		M_1&:=|\Re(A\nabla \phi, \nabla(\operatorname{id}-I_H)(\eta\phi))_{\Omega}|,\\
		M_2&:=|\Re(A\nabla \phi, \nabla I_H(\eta\phi))_{\Omega}|,\\
		M_3&:=|\Re(A\nabla \phi, (\nabla\eta)\phi)_{\Omega}|.
		\end{align*}
		
		Since $w:=(\operatorname{id}-I_H)(\eta\phi)\in W$, the idealized corrector problem \eqref{eq:correc} and the fact that $w$ has only support outside $T$ imply that $\CB(\phi, w)=\CB_T(v_H, w)=0$.
		Therefore, we obtain 
		\[M_1=|\Re(A\nabla \phi, \nabla w)|\leq \Bigl|\Re\Bigl(\CB(\phi, w)+k^2(\phi, w)\Bigr)\Bigr|=|k^2\Re(\phi, w)|.\]
		Hence, the stability and approximation estimates \eqref{eq:IHstab} and \eqref{eq:IHapprox} for $I_H$, the properties of $\eta$ , and the resolution condition \eqref{eq:resolcondH} give
		\begin{align*}
		M_1&\lesssim  \varepsilon^{-2}k^2 H^2(|\phi|^2_{1, A, \Omega\setminus\UN^{m}(T)}+\|\nabla \eta\|_{\UN(\CR)}\|\phi\|_{0, A, \UN(\CR)}|\phi|_{1, A, \UN^2(\CR)})\\
		&\lesssim \varepsilon^{-2}k^2 H^2(|\phi|^2_{1, A, \Omega\setminus\UN^{m}(T)}+|\phi|^2_{1, A, \UN^2(\CR)})\\
		&\lesssim \frac12 (|\phi|^2_{1, A, \Omega\setminus\UN^m(T)}+|\phi|^2_{1,A,\UN^2(\CR)}),
		\end{align*}
		where the first term can be hidden on the left-hand side.
		
		Because of $\supp(I_H(\eta\phi))\subset \UN(\CR)$, the properties of $I_H$ and the above estimate for $|\eta\phi|_{1, A}$ lead to
		\begin{align*}
		M_2\lesssim |\phi|_{1, A, \UN(\CR)}|\eta\phi|_{1, A, \UN(\CR)}\lesssim |\phi|^2_{1, A, \UN^2(\CR)}.
		\end{align*}
		Finally, the properties of $\eta$ and the approximation result \eqref{eq:IHapprox} for $I_H$ show that
		\[M_3\lesssim |\phi|_{1, A, \CR}|\phi|_{1, A, \UN(\CR)}.\]
		
		All in all, this gives
		\[|\phi|^2_{1, A, \Omega\setminus\UN^m(T)}\leq C|\phi|^2_{1, A, \UN^2(\CR)}\]
		with a constant $C$ independent of $H$, $m$, $k$, and $\varepsilon$. We recall that $\UN^2(\CR)=\UN^m(T)\setminus\UN^{m-5}(T)$.
		Because of 
		\[|\phi|^2_{1, A, \Omega\setminus\UN^m(T)}+|\phi|^2_{1, A, \UN^m(T)\setminus \UN^{m-5}(T)}=|\phi|^2_{1, A, \UN^{m-5}(T)}\]
		we obtain
		\[|\phi|^2_{1, A, \Omega\setminus\UN^m(T)}\leq (1+2C^{-1})^{-1}|\phi|_{1, A, \Omega\setminus \UN^{m-5}(T)}.\]
		Note that $(1+2C^{-1})^{-1}<1$, so that a repeated application of the above argument and algebraic manipulations finish the proof. Note in particular that $\tilde{\beta}$ is independent of $H$, $m$, $k$, and $\varepsilon$ because $C$ is independent of these quantities.
	\end{proof}
	
	The idea for the proof of Theorem \ref{thm:correcerror} is now that the ideal corrector is only truncated in $\Omega\setminus\UN^m(T)$, where we know by the previous Proposition that the contribution is exponentially small.
	
	\begin{proof}[Proof of Theorem \ref{thm:correcerror}]
		We define the cutoff function $\eta\in\CS^1(\CT_H)$ via
		\[\eta=1\quad\text{in}\quad\UN^{m-2}(T),\qquad\eta=0\quad \text{in}\quad \Omega\setminus \UN^{m-1}(T).\]
		With the ellipticity and continuity of $\CB$ over $W$, we deduce C\'ea's Lemma
		\[|(\CQ_{T, m}-\CQ_{T})v_H|_{1, A}\lesssim 
		\inf_{w_{T, m}\in W(\UN^m(T))}|\CQ_{T}v_H-w_{T, m}|_{1, A}.\]
		Inserting $w_{T, m}=(\operatorname{id}-I_H)(\eta\CQ_{T}v_H)$ and using stability and approximation properties of $I_H$ yields -- similar to the proof of Proposition \ref{prop:correcdecay} --
		\begin{equation}\label{eq:elementcorrector}
		|(\CQ_{T, m}-\CQ_{T})v_H|_{1, A}\lesssim
		|\CQ_T v_H|_{1, A, \Omega\setminus \UN^m(T)}.
		\end{equation}
		Applying \eqref{eq:correcdecay} results in an exponential decay result for the error between the element correctors $\CQ_T $ and $\CQ_{T, m}$.
		
		To obtain the global estimate, we set $z:=(\CQ-\CQ_m)v_H$ and $z_T:=(\CQ_{T}-\CQ_{T, m})v_H$ and define the cutoff function $\eta\in\CS^1(\CT_H)$ via
		\[\eta=0\quad\text{in}\quad\UN^{m+1}(T), \qquad\eta=1\quad \text{in}\quad \Omega\setminus \UN^{m+2}(T).\]
		The ellipticity of $\CB$ yields
		\begin{align*}
		|z|^2_{1, A}\lesssim \sum_{T\in \CT_H}|\CB(z_T, z)|\leq \sum_{T\in \CT_H}|\CB(z_T, (1-\eta)z)|+|\CB(z_T, (\operatorname{id}-I_H)(\eta z))|+|\CB(z_T, I_H(\eta z))|.
		\end{align*}
		The second term vanishes because $(\operatorname{id}-I_H)(\eta z)\in W$ with support outside $\UN^m(T)$.
		The function $(1-\eta)z$ vanishes on $S:=\{\eta=1\}$ and we obtain with the scaling properties of $\eta$, \eqref{eq:IHapprox}, \eqref{eq:IHstab}, and the resolution condition \eqref{eq:resolcondH} that
		\begin{align*}
		|\CB(z_T, (1-\eta)z)|&\lesssim 
		\|(1-\eta)z\|_{1,k, A, \Omega\setminus S}|z_T|_{1, A}\\
		&\lesssim 
		(k\|z\|_{\Omega\setminus S}+|z|_{1, A, \Omega\setminus S}+H^{-1}\|z\|_{0, A, \Omega\setminus S})|z_T|_{1, A}\\
		&\lesssim 
		|z|_{1, A, \Omega\setminus S}|z_T|_{1, A}.
		\end{align*}
		
		$I_H(\eta z)$ vanishes on $\Omega\setminus \UN(\supp(1-\eta))$. Hence, we infer with the properties of $\eta$, \eqref{eq:IHapprox}, \eqref{eq:IHstab}, and the resolution condition \eqref{eq:resolcondH} that
		\begin{align*}
		&\!\!\!\!|\CB(z_T, I_H(\eta z))|\\
		&\lesssim  
		\|I_H(\eta z)\|_{1, A, k, \UN(\supp(1-\eta))}|z_T|_{1, A}\\
		&\lesssim 
		\Bigl(|z|_{1, A, \UN^2(\supp(1-\eta))}+k\|\eta z-I_H(\eta z)\|_{\UN(\supp(1-\eta))}+k\|\eta z\|_{\UN(\supp(1-\eta))}\Bigr)|z_T|_{1, A}\\
		&\lesssim 
		\Bigl(|z|_{1, A, \UN^2(\supp(1-\eta))}+kH\varepsilon^{-1}|\eta z|_{1, A, \UN^2(\supp(1-\eta))}+kH\varepsilon^{-1}|z|_{1, A, \UN(\supp(1-\eta))}\Bigr)|z_T|_{1, A}\\
		&\lesssim 
		|z|_{1, A, \UN^2(\supp(1-\eta))}|z_T|_{1, A}.
		\end{align*}
		All in all, summation over all $T$ and Cauchy inequality yield with the finite overlap of patches
		\[|z|_{1, A}^2\lesssim
		C_{\operatorname{ol}, m}^{1/2}|z|_{1, A}\Biggl(\sum_{T\in \CT_H}|z_T|^2_{1, A}\Biggr)^{1/2}.\]
		Combination with \eqref{eq:elementcorrector} finishes the proof.
	\end{proof}
	
	\subsection{Proof of the well-posedness of the method and of the error estimates}
	\label{subsec:proofestimates}
	
	We first prove the discrete inf-sup condition which implies that the solution of the multiscale method is indeed well-defined.
	\begin{proof}[Proof of Theorem \ref{thm:LODwellposed}]
		Let $v_H\in V_H$. From the inf-sup condition for the model problem we infer that there exists $\psi\in H^1(\Omega)$ with $\|\psi\|_{1, A, k}=1$ such that
		\[\Re\CB((\operatorname{id}+\CQ) v_H, \psi)\geq \gamma(\varepsilon, k)\|(\operatorname{id}+\CQ) v_H\|_{1, A, k}.\]
		Define now $\psi_H:=I_H\psi$. We have
		\begin{align*}
		&\!\!\!\!\CB((\operatorname{id}+\CQ_m)v_H, (\operatorname{id}+\CQ^*_m)\psi_H)\\
		&=\CB((\operatorname{id}+\CQ_m)v_H, (\operatorname{id}+\CQ^*_m)\psi_H-(\operatorname{id}+\CQ^*)\psi_H) + \CB((\operatorname{id}+\CQ_m)v_H, (\operatorname{id}+\CQ^*)\psi_H)\\
		&=\CB((\operatorname{id}+\CQ_m)v_H, (\CQ_m^*-\CQ^*)\psi_H)+ \CB((\operatorname{id}+\CQ)v_H, (\operatorname{id}+\CQ^*)\psi_H).
		\end{align*}
		Since $\CQ^*$ is a projection onto $W$, we have $(\operatorname{id}+\CQ^*)I_H\psi=(\operatorname{id}+\CQ^*)\psi$.
		The solution properties of $\CQ^*$ and \eqref{eq:correc} of $\CQ$ imply $\CB((\operatorname{id}+\CQ)v_H, (\operatorname{id}+\CQ^*)\psi_H)=\CB((\operatorname{id}+\CQ)v_H, \psi)$.
		Hence, the continuity of $\CB$ and \eqref{eq:correcerror} imply
		\begin{align*}
		&\!\!\!\!\Re \CB((\operatorname{id}+\CQ_m)v_H, (\operatorname{id}+\CQ^*_m)\psi_H)\\
		&\geq \gamma(\varepsilon, k)\|(\operatorname{id}+\CQ)v_H\|_{1, A, k}-
		\|(\operatorname{id}+\CQ_m)v_H\|_{1, A, k}|(\CQ_m^*-\CQ^*)\psi_H|_{1, A}\\
		& \gtrsim \gamma(\varepsilon, k)\|(\operatorname{id}+\CQ)v_H\|_{1, A, k}-
		C_{\operatorname{ol}, m}^{1/2}\beta^m\|(\operatorname{id}+\CQ_m)v_H\|_{1, A, k}\|\psi_H\|_{1, A, k}.
		\end{align*}
		The stability of $I_H$ yields $\|\psi_H\|_{1, A, k}\lesssim \|\psi\|_{1, A, k}=1$ and
		\[\|v_H\|_{1, A, k}=\|I_H((\operatorname{id}+\CQ) v_H)\|_{1, A, k}\lesssim \|(\operatorname{id}+\CQ) v_H\|_{1, A, k}.\]
		Moreover, we have the following stability for $\CQ_m$:
		\[\|\CQ_m v_H\|_{1, A, k}\lesssim 
		\|v_H\|_{1, A, k}.\]
		This finally gives 
		\begin{align*}
		\Re \CB((\operatorname{id}+\CQ_m)v_H, (\operatorname{id}+\CQ^*_m)\psi_H)\gtrsim (\gamma(\varepsilon, k)-
		C_{\operatorname{ol}, m}^{1/2}\beta^m)\|v_H\|_{1, A, k}\|\psi_H\|_{1, A, k},
		\end{align*}
		which together with the oversampling condition \eqref{eq:oversamplcond} finishes the proof.
	\end{proof}
	
	We can now prove the a priori error estimates which shed alight onto the approximation properties of the multiscale method.
	\begin{proof}[Proof of Theorem \ref{thm:LODapriori}]
		\emph{Proof of \eqref{eq:errorfineenergy}:}
		Denote $e:=u-u_{\operatorname{LOD}, m}=u-(\operatorname{id}+\CQ_m)u_{H, m}$ and define $e_{H, m}:=(\operatorname{id}+\CQ_m)I_H e=(\operatorname{id}+\CQ_m)(I_H u-u_{H, m})$.
		The triangle inequality gives
		\[\|e\|_{1, A, k}\leq \|e-e_{H, m}\|_{1, A, k}+\|e_{H, m}\|_{1, A, k}.\]
		We will show that under the oversampling condition \eqref{eq:oversamplcond} the second term can be bounded by the first term.
		
		Let $z_H\in V_H$ be the solution of the adjoint problem
		\[\CB((\operatorname{id}+\CQ_m)v_H, (\operatorname{id}+\CQ_m^*)z_H)=((\operatorname{id}+\CQ_m)v_H, e_{H, m})_{1, A, k}\qquad \forall v_H\in V_H.\]
		Then we obtain with the orthogonality of $W$ and $(\operatorname{id}+\CQ)V_H$ and the Galerkin orthogonality $\CB(e, (\operatorname{id}+\CQ_m^*)\psi_H)=0$ for all $\psi_H \in V_H$ that
		\begin{align*}
		\|e_{H, m}\|^2_{1, A, k}&=\CB(e_{H, m}, (\operatorname{id}+\CQ_m^*)z_H)\\
		&=\CB(e_{H, m}, (\CQ_m^*-\CQ^*)z_H)+\CB(e_{H, m}, (\operatorname{id}+\CQ^*)z_H)\\
		&=\CB(e_{H, m}, (\CQ_m^*-\CQ^*)z_H)+\CB(e, (\operatorname{id}+\CQ^*)z_H)\\
		&=\CB(e_{H, m}, (\CQ_m^*-\CQ^*)z_H)+\CB(e, (\CQ^*-\CQ_m^*)z_H)\\
		&=\CB(e-e_{H, m}, (\CQ^*-\CQ_m^*)z_H)
		\end{align*}
		The solution $z_H$ of the adjoint problem fulfills the following stability
		\[\|z_H\|_{1, A, k}\leq 
		\gamma^{-1}_{\operatorname{LOD}}\|e_{H, m}\|_{1, A, k}.\]
		Using this stability and the corrector error \eqref{eq:correcerror} results in
		\[\|e_{H, m}\|^2_{1, A, k}\lesssim 
		C_{\operatorname{ol}, m}^{1/2}\gamma^{-1}_{\operatorname{LOD}}\|e-e_{H, m}\|_{1, A, k}\|e_{H, m}\|_{1, A, k}.\]
		Dividing by $\|e_{H, m}\|_{1, A, k}$ on both sides and using the oversampling condition \eqref{eq:oversamplcond} gives the bound $\|e_{H, m}\|_{1, A k}\lesssim \|e-e_{H, m}\|_{1, A, k}$.	
		Second, we estimate the error $e-e_{H, m}=u-(\operatorname{id}+\CQ_m)I_H u\in W$.
		With the ellipticity of $\CB$ over $W$, the $\CB$-orthogonality between $W$ and $(\operatorname{id}+\CQ)V_H$ and Galerkin orthogonality we obtain
		\begin{align*}
		\|e-e_{H, m}\|^2_{1, A, k}&\lesssim \CB(e-e_{H, m}, e-e_{H, m})\\
		&=\CB(u-(\operatorname{id}+\CQ)I_H u, e-e_{H, m}))+\CB((\CQ-\CQ_m)I_H u, e-e_{H, m})\\
		&=\CB(u, e-e_{H, m})+\CB((\CQ-\CQ_m)I_H u, e-e_{H, m})\\*
		&=(f, e-e_{H, m})+\CB((\CQ-\CQ_m)I_H u, e-e_{H, m}).
		\end{align*}
		The last term can be estimated employing \eqref{eq:correcerror} as
		\begin{align*}
		|\CB((\CQ-\CQ_m)I_H u, e-e_{H, m})|\lesssim 
		C_{\operatorname{ol}, m}^{1/2}\beta^m\|u\|_{1, A, k}\|e-e_{H, m}\|_{1, A, k},
		\end{align*}
		where we also used the continuity of $\CB$ and the stability of $I_H$.
		$\|u\|_{1, A, k}$ can be estimated against the data with the inf-sup constant of the model problem.
		For the first term, we can insert $I_H(e-e_{H, m})$ and use the approximation properties of $I_H$ to obtain
		\[|(f, e-e_{H, m})|\lesssim H\|A^{-1/2}f\|\,\|e-e_{H, m}\|_{1, A, k}.\]
		
		\emph{Proof of \eqref{eq:errorfineL2w}:} We again abbreviate $e:=u-u_{\operatorname{LOD}, m}$ and consider the following two adjoint problems:
		(i) Find $z\in H^1(\Omega)$  such that
		\[\CB(\psi, z)=(\psi, e)_{0, A}\qquad \forall \psi\in H^1(\Omega);\]
		(ii) Find $z_{H, m}\in V_H$  such that
		\[\CB((\operatorname{id}+\CQ_m)\psi_H, (\operatorname{id}+\CQ_m^*)z_{H, m})=((\operatorname{id}+\CQ_m)\psi_H, e)_{0, A}\qquad \forall \psi\in H^1(\Omega).\]
		Set $z_{\operatorname{LOD}, m}:=(\operatorname{id}+\CQ_m^*)z_{H, m}$.
		Then we deduce with Galerkin orthogonality
		\[\|e\|_{0, A}^2=|\CB(e, z)|=|\CB(e, z-z_{\operatorname{LOD}, m})|\lesssim
		\|e\|_{1, A, k}\|z-z_{\operatorname{LOD}, m}\|_{1, A, k}. \]
		Note that $z$ and $z_{H, m}$ are solutions to adjoint problems with a volume term on the right-hand side of the form $Ae$.
		Hence, we deduce similar to the above estimate for $\|e\|_{1, A, k}$ that
		\[\|z-z_{\operatorname{LOD}, m}\|_{1, A, k}\lesssim H\|A^{1/2}e\|+
		C_{\operatorname{ol}, m}^{1/2}\beta^m\gamma^{-1}(\varepsilon, k)\|A e\|.\]
		Combination of the foregoing estimates finishes the proof.
		
		\emph{Proof of \eqref{eq:errorcoarseL2w}:} 
		Let $u_{\operatorname{LOD}}$ be the solution to the idealized variant of \eqref{eq:LOD}, i.e., with $\UN^m(T)=\Omega$ for all elements $T$, as introduced at the beginning of Section \ref{subsec:erroranalysis}.
		Because of \eqref{eq:IHapprox} and the definition of the norms in \eqref{eq:L2Anorm} and \eqref{eq:H1Anorm}, we obtain
		\begin{align*}
		\|u-u_{H, m}\|_{0, A}&\leq \|u-I_H u\|_{0, A}+\|I_H u-u_{H, m}\|_{0, A}\\
		&\lesssim H|u-I_Hu|_{1, A}+\|I_H u -u_{H, m}\|_{1, A, k}.
		\end{align*}
		The projection property of $I_H$ and its stability \eqref{eq:IHstab} imply for the first term
		\[|u-I_Hu|_{1, A}\lesssim \inf_{v_H\in V_H}|u-v_H|_{1, A}.\]
		To estimate the second term, we abbreviate $e_H:=I_H u-u_{H,m}\in V_H$ and note that by the definition of $\CQ$ and the stability \eqref{eq:IHstabenergy}, we have
		\[\|e_H\|_{1, A, k}=\|I_H(\operatorname{id}+\CQ)e_H\|_{1, A, k}\lesssim \|(\operatorname{id}+\CQ)e_H\|_{1, A,k}.\]
		Because of the discrete inf-sup condition \eqref{eq:LODinfsup} (which also holds for the idealized version of the LOD), there exist $z_H\in V_H$ with $\|z_H\|_{1, A, k}=1$ such that
		\[\|(\operatorname{id}+\CQ)e_H\|_{1, A,k}\lesssim \gamma_{\operatorname{LOD}}^{-1}|\CB((\operatorname{id}+\CQ)e_H, (\operatorname{id}+\CQ^*)z_H)|.\]
		Proceeding in a similar way as for the estimate of $e_{H,m}$ above, we obtain with the definition of $\CQ$ and Galerkin orthogonality that
		\begin{align*}
		\|e_H\|_{1, A, k}&\lesssim \gamma_{\operatorname{LOD}}^{-1}|\CB((\operatorname{id}+\CQ)e_H, (\operatorname{id}+\CQ^*)z_H)|\\
		&=\gamma_{\operatorname{LOD}}^{-1}|\CB((\operatorname{id}+\CQ)I_H u -(\operatorname{id}+\CQ)u_{H, m}, (\operatorname{id}+\CQ^*)z_H)|\\
		&=\gamma_{\operatorname{LOD}}^{-1}|\CB(u-(\operatorname{id}+\CQ_m)u_{H,m}, (\operatorname{id}+\CQ^*)z_H)|\\
		&=\gamma_{\operatorname{LOD}}^{-1}|\CB(u-u_{\operatorname{LOD}, m}, (\CQ^*-\CQ_m)z_H)|\\
		&\lesssim \gamma_{\operatorname{LOD}}^{-1}C_{\operatorname{ol}, m}^{1/2}\beta^m\|u-u_{\operatorname{LOD}, m}\|_{1, A, k},
		\end{align*}
		where we used \eqref{eq:correcerror} in the last estimate.
		The remaining error $\|u-u_{\operatorname{LOD}, m}\|_{1, A, k}$ has been estimated in \eqref{eq:errorfineenergy}.
	\end{proof}

	\section{Numerical experiments}
	\label{sec:numexp}
	We investigate the method and the a priori estimates in several experiments.
	The convergence history plots display the relative error in the indicated norm versus the (coarse) mesh size $H$.
	We consider the LOD solution $u_{\operatorname{LOD}, m}$ defined in \eqref{eq:uLODm} as well as its FE part $u_{H,m}$, cf. \eqref{eq:LOD}.
	For comparison, we also compute a standard FE solution (denoted as P1FEM in the convergence plots) on $V_H$ as well as the best approximation in $V_H$ with respect to the indicated norm (denoted as P1-best in the convergence plots).
	The solution plots display the real part with a color map truncated to the interval $[-2, 2]$ to visualize the wave behavior outside of the scatterer, which is our main interest.
	We use a uniform simplicial mesh on the unit square $\Omega=(0,1)^2$ in all our experiments and compute all correctors, i.e., we do not exploit the periodic structure as discussed in Section \ref{subsec:implremark}.
	As right-hand side we use 
	\[f(x)=
	\begin{cases}
	10000\:\exp\Bigl(-\frac{1}{1-\bigl(\frac{|x-x_0|}{0.05}\bigr)^2}\Bigr)\qquad \text{if}\quad\frac{|x-x_0|}{0.05}<1,\\
	0\qquad \text{else},
	\end{cases}\]
	with varying center $x_0$.
	We compare results for the $A$-weighted interpolation operator $I_H$ \eqref{eq:IHweighted} as well as its unweighted variant $I_H^1$ \eqref{eq:IHunweighted}.
	
	\subsection{Periodic structure with Mie resonances}
	We consider a periodic scatterer in the domain $(0.25, 0.75)^2$ with
	\[\Omega_\varepsilon=(0.25, 0.75)^2\cap \bigcup_{j\in \gz^2}\varepsilon(j+(0.25, 0.75)^2).\]
	We choose the center for $f$ as $x_0=(0.125, 0.5)$ and the wave number $k=9$. 
	This choice of $k$ is connected to a negative-valued effective wave number $k^2\mu$ in homogenization theory and this unusual behavior is induced by (Mie) resonances in the small inclusions, cf. \cite{OV16hmmhelmholtz}.
	Therefore, we expect that an approximation of the exact solution might be hard in that case and consider it as a good test for the multiscale method.
	Note that this setting fulfills our assumptions, in particular $\Omega_\varepsilon$ is uniformly bounded away from $\partial\Omega$.
A reference solution (still denoted $u$) is computed with standard finite elements with mesh size $h=2^{-9}$ which is also the fine mesh used for the discretization of the corrector problems  as explained in Section \ref{subsec:implremark}.
	The convergence history in $H$ for the errors of $u_{\operatorname{LOD}, m}$ and $u_{H, m}$ with respect to the reference solution $u$ is studied for patch sizes $m=1,2,3$.
	
	\begin{figure}
		\includegraphics[width=0.47\textwidth, trim=42mm 93mm 55mm 98mm, clip=true, keepaspectratio=false]{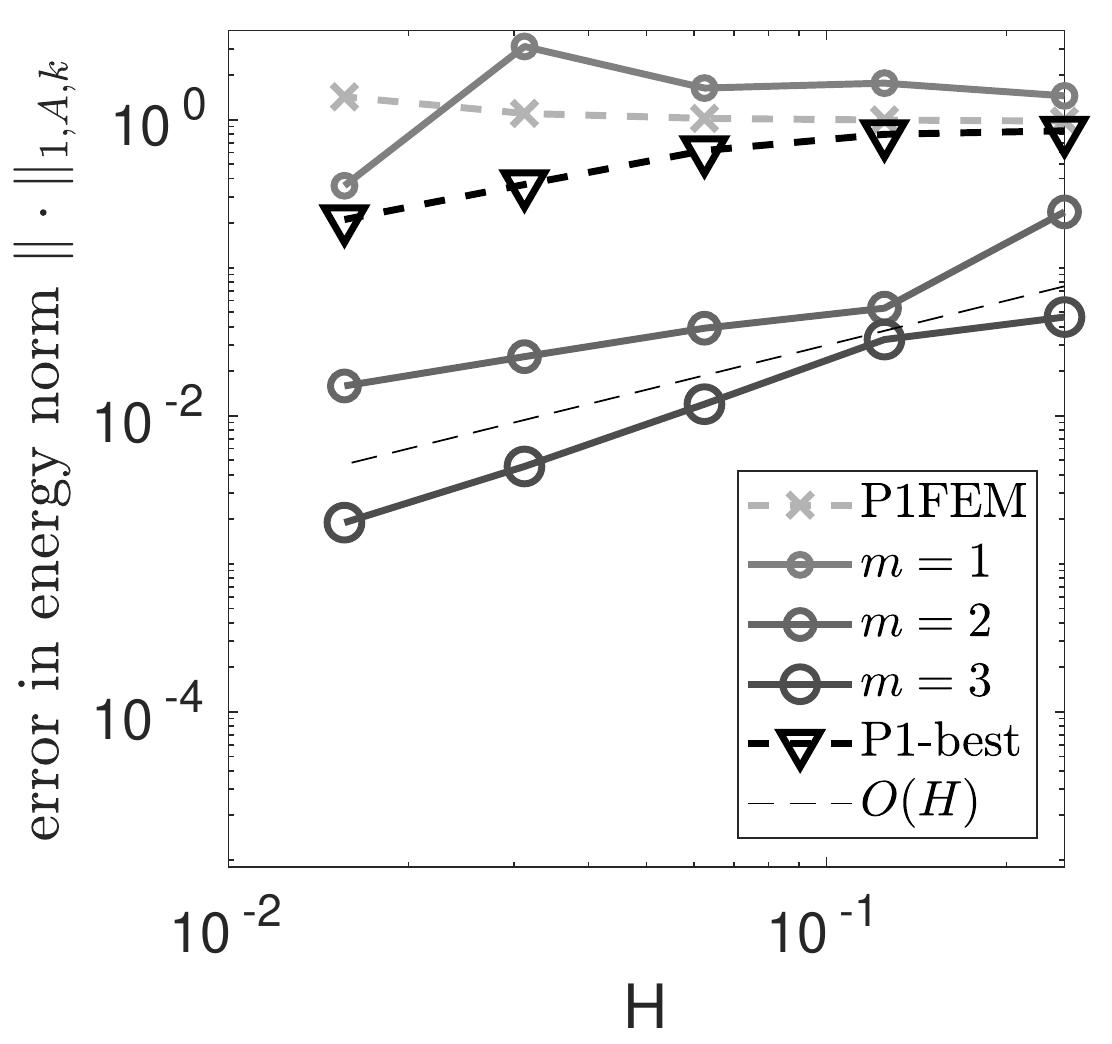}%
		\hspace{5ex}%
		\includegraphics[width=0.47\textwidth, trim=42mm 93mm 55mm 98mm, clip=true, keepaspectratio=false]{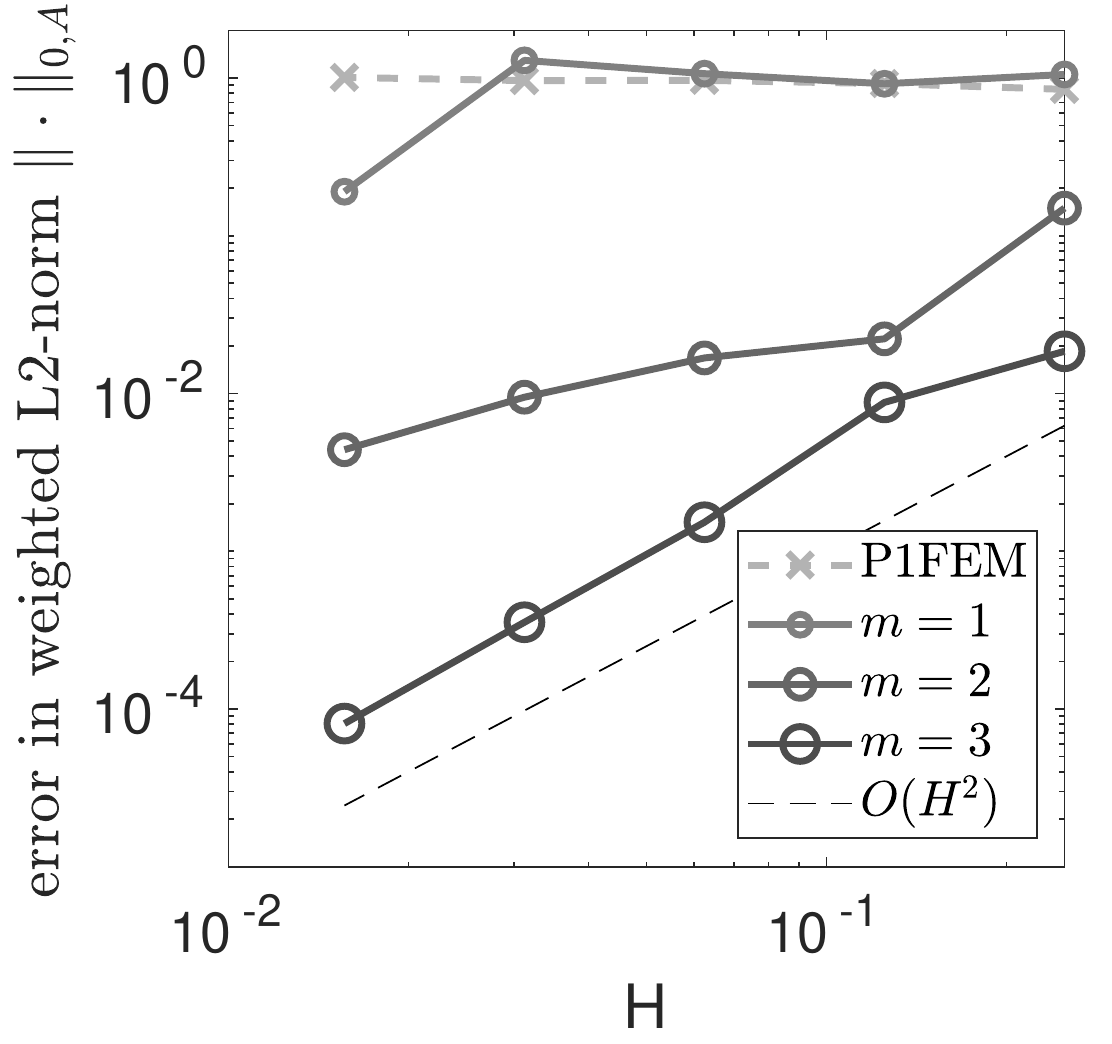}
		\caption{Convergence history of the (relative) error $u-u_{\operatorname{LOD}, m}$ for $\varepsilon=2^{-3}$. In the left figure, P1-best denotes the best approximation in $V_H$ w.r.t. $\|\cdot\|_{1,A,k}$.}	
		\label{fig:upscalederror-N8}
	\end{figure}
	
	For $\varepsilon=2^{-3}$,  
	the results obtained with $I_H$ and $I_H^1$ are very similar so that we restrict the confirmation of the convergence rates to $I_H^1$.
	Figure \ref{fig:upscalederror-N8} shows the convergence histories for $u-u_{\operatorname{LOD}, m}$ and we verify the linear rate in the energy norm (and the standard $L^2$ norm) as well as the quadratic rate in the weighted $L^2$ norm.
	A patch size of $m=2,3$ seems to be sufficient.
	We observe a clear superiority of the multiscale method over a standard finite element method on the coarse mesh, which fails to yield good approximations for these values of $H$.
	The standard FEM even deviates significantly from the FE best approximations with respect to the energy norm and weighted $L^2$-norm (computed by projecting the reference solution onto $V_H$). The FE best approximations are outperformed by the multiscale method which takes fine-scale features into account.
	
	\begin{figure}
		\includegraphics[width=0.47\textwidth, trim=42mm 93mm 55mm 98mm, clip=true, keepaspectratio=false]{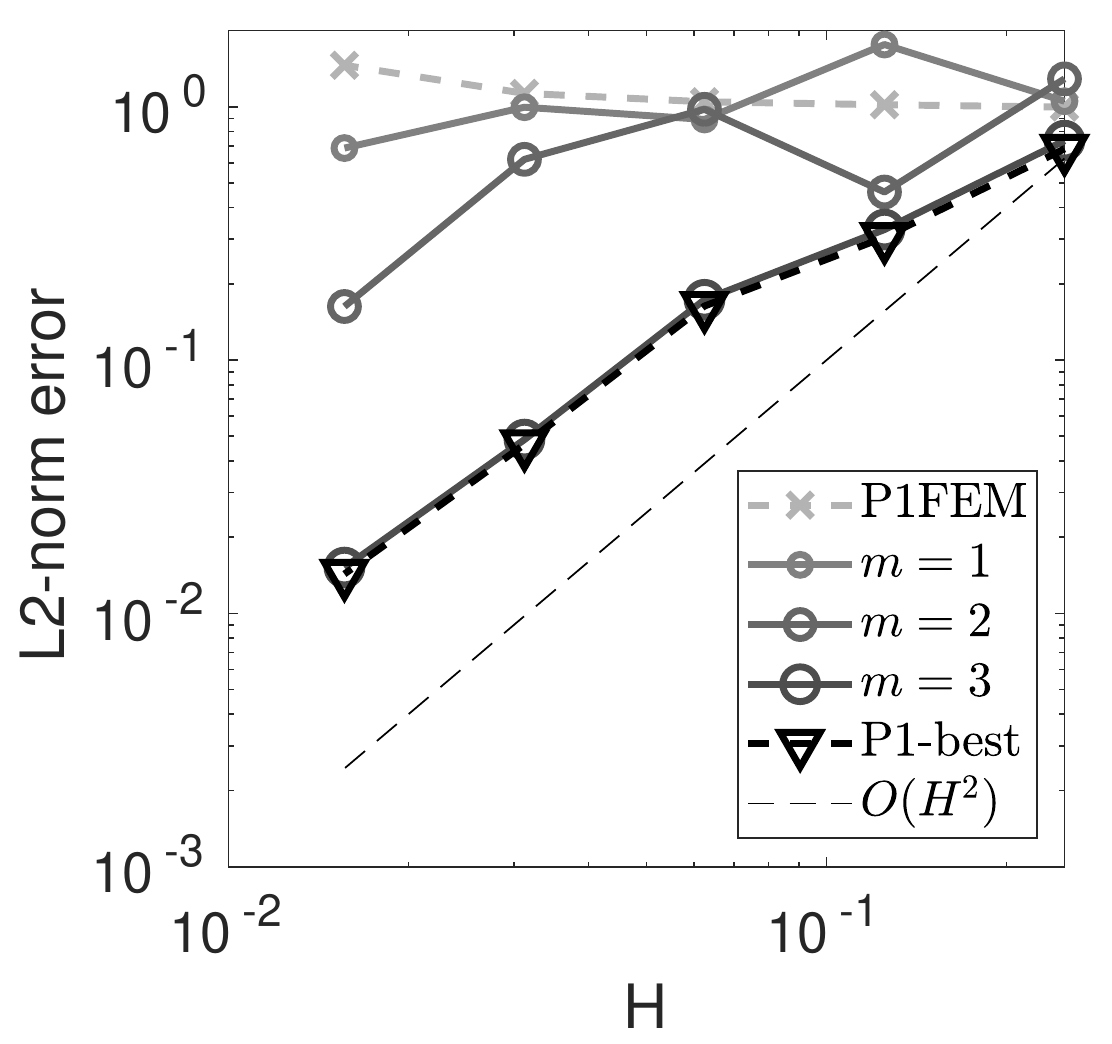}%
		\hspace{5ex}%
		\includegraphics[width=0.47\textwidth, trim=42mm 93mm 55mm 98mm, clip=true, keepaspectratio=false]{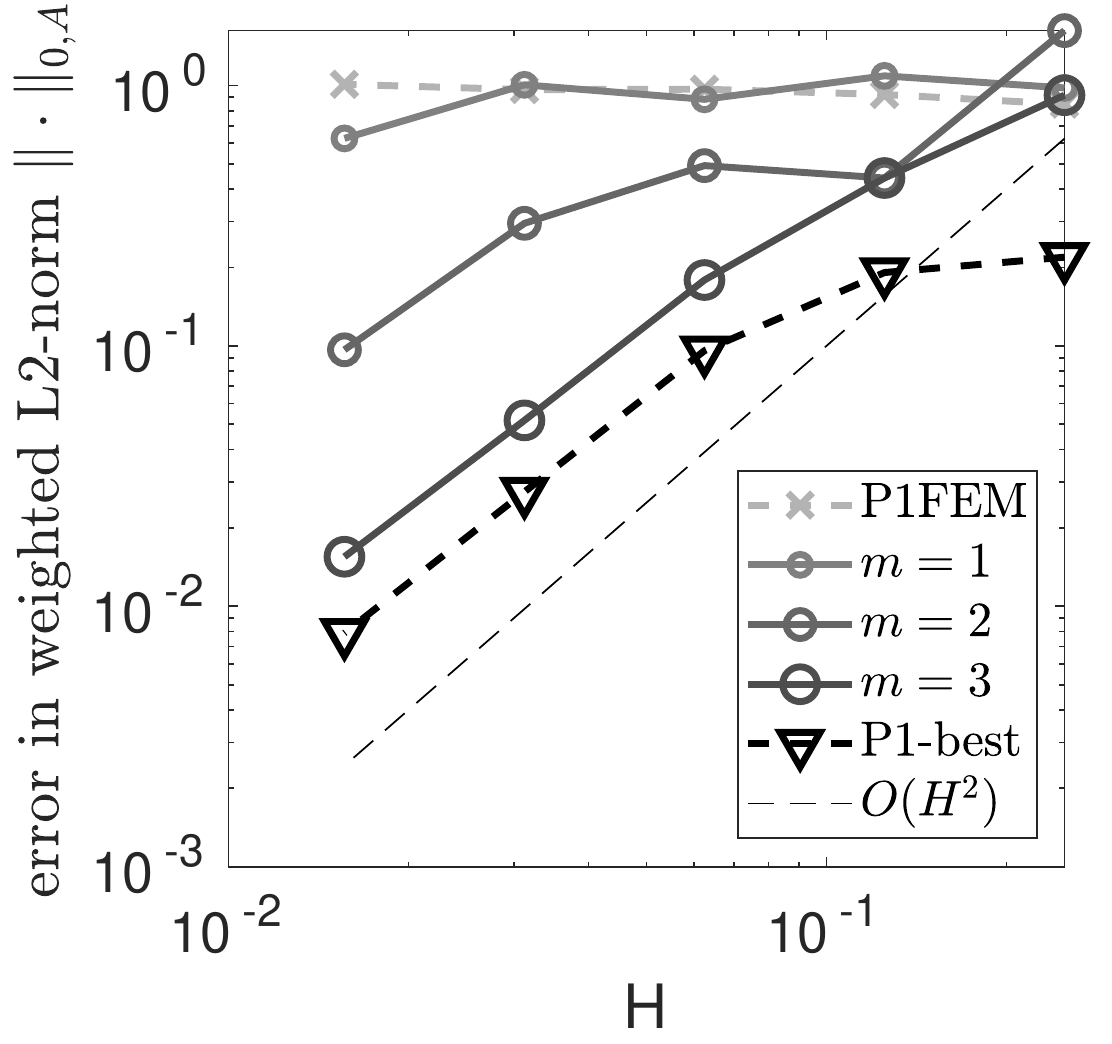}%
		\caption{Convergence history of the (relative) error $u-u_{H, m}$ for $\varepsilon=2^{-3}$. Note that P1-best stands for the best approximation in $V_H$ w.r.t. the $L^2$-norm (left) and the weighted $L^2$-norm $\|\cdot\|_{0,A}$ (right), respectively.}	
		\label{fig:macroerror-N8}
	\end{figure}
	
	In this context it is very interesting to analyze the error $u-u_{H, m}$ of the FE part $u_{H, m}$ of the multiscale method in Figure \ref{fig:macroerror-N8}.
	In the standard as well as the weighted $L^2$ norm, we observe a convergence which closely follows the behavior of the FE best approximations in the space $V_H$ with respect to the corresponding norms. 
	This clearly underlines that there exists a macroscopic approximation to the exact solution which is good in an $L^2$ sense.
	While the multiscale method is able to (almost) find this best approximation, the FEM (using the same approximation space) fails completely.
	Note that we observe a numerical convergence rate of about $2$ for the error $u-u_{H, m}$ in this example.
	This is explained by higher regularity of the exact solution, which results in an order $H$ for the best-approximation error $\inf_{v_H\in V_H}|u-v_H|_{1, A,k}$ as discussed after Theorem \ref{thm:LODapriori}.
	
	\begin{figure}
		\includegraphics[width=0.47\textwidth, trim=42mm 93mm 55mm 98mm, clip=true, keepaspectratio=false]{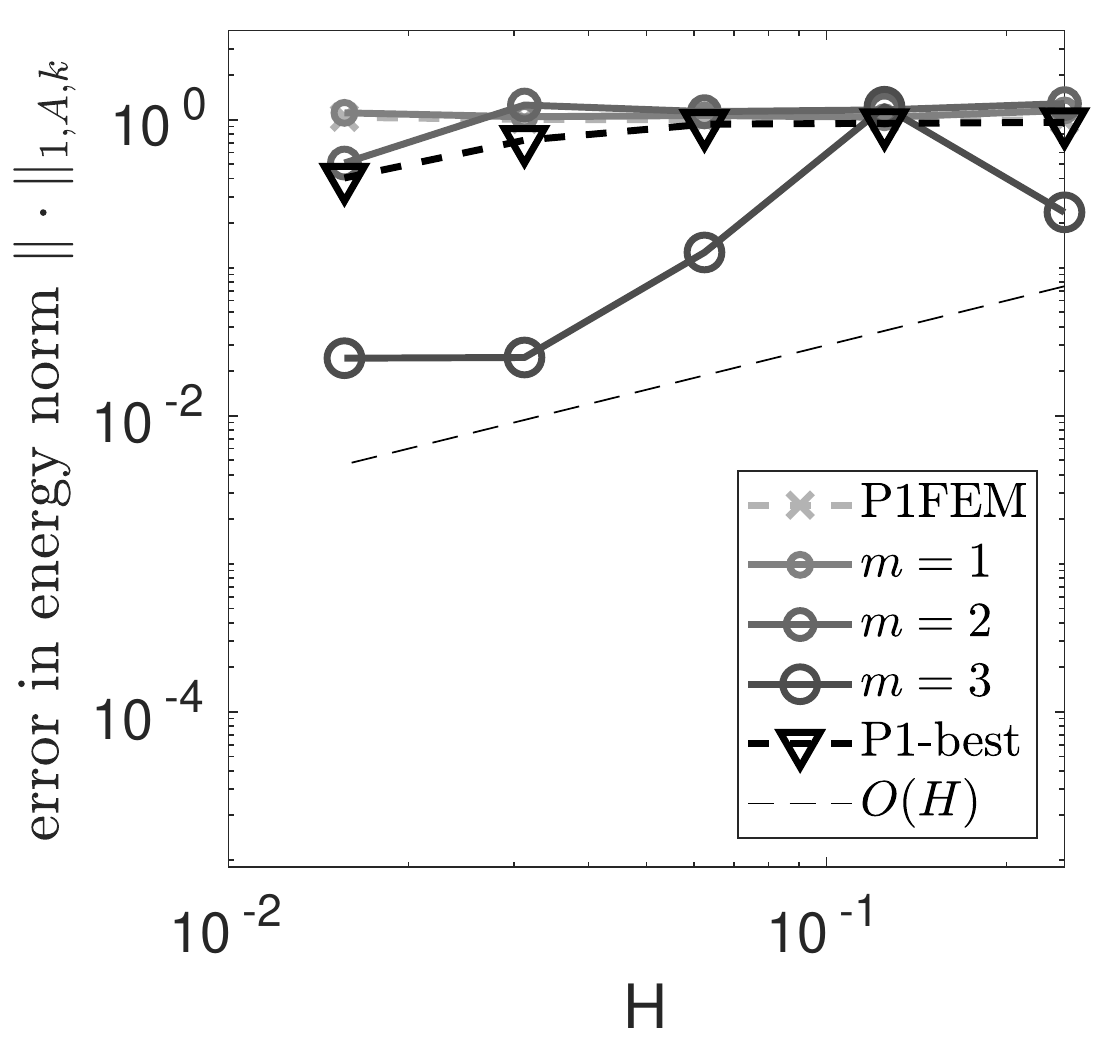}%
		\hspace{5ex}%
		\includegraphics[width=0.47\textwidth, trim=42mm 93mm 55mm 98mm, clip=true, keepaspectratio=false]{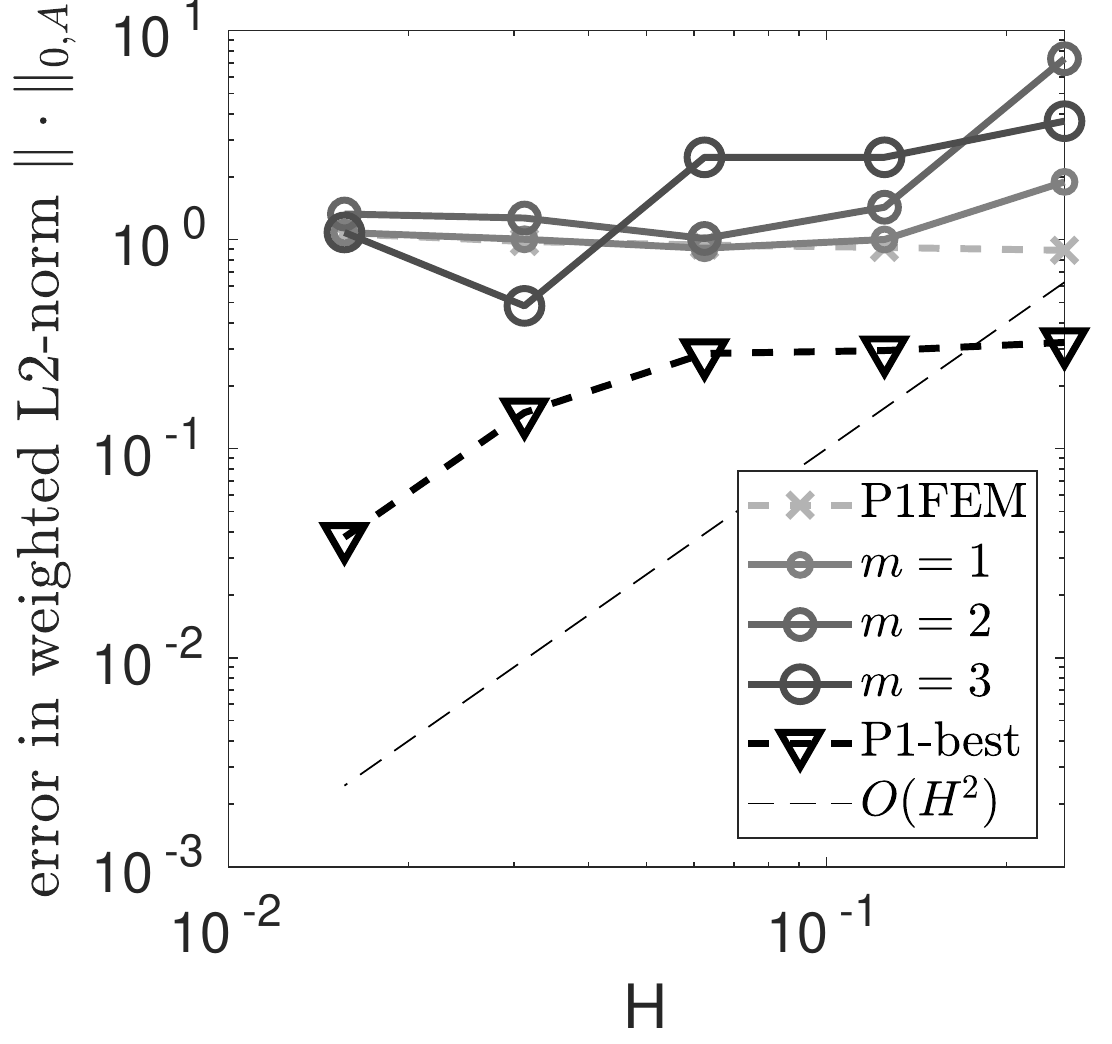}\\
		\includegraphics[width=0.47\textwidth, trim=42mm 93mm 55mm 98mm, clip=true, keepaspectratio=false]{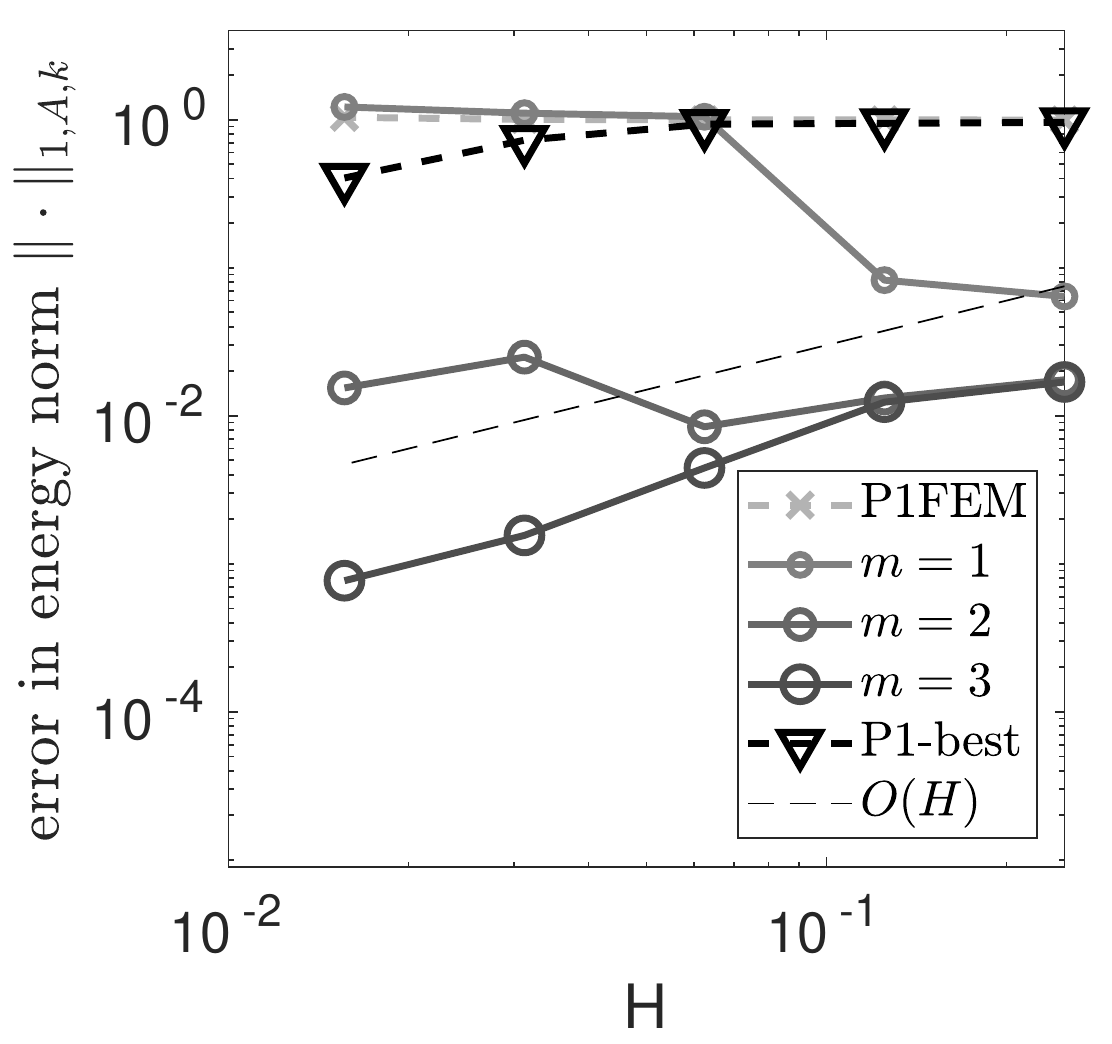}%
		\hspace{5ex}%
		\includegraphics[width=0.47\textwidth, trim=42mm 93mm 55mm 98mm, clip=true, keepaspectratio=false]{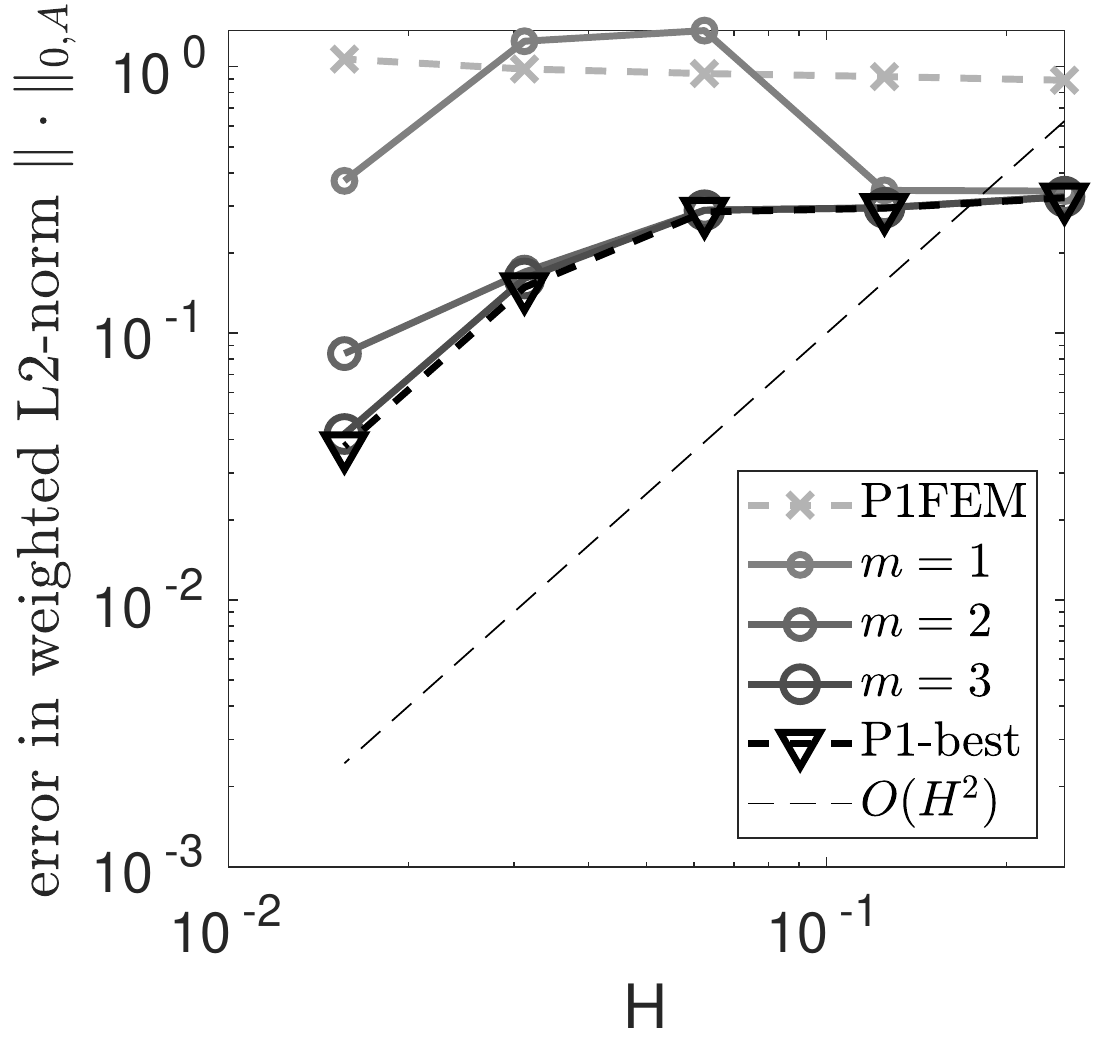}
		\caption{Convergence history of the (relative) errors for unweighted (top) and weighted (bottom) interpolation operator, both for $\varepsilon=2^{-4}$. Note that P1-best stands for the best approximation in $V_H$ w.r.t. $\|\cdot\|_{1,A,k}$ (left) and $\|\cdot\|_{0,A}$ (right), respectively.}	
		\label{fig:error-N16}
	\end{figure}

	For $\varepsilon=2^{-4}$, we compare $\|u-u_{\operatorname{LOD}, m}\|_{1, A, k}$ and $\|u-u_{H, m}\|_{0, A}$ for the unweighted interpolation operator $I_H^1$ \eqref{eq:IHunweighted} and its $A$-weighted variant $I_H$ \eqref{eq:IHweighted}.
	As Figure \ref{fig:error-N16} shows, only the weighted interpolant succeeds to yield good approximations (for patch size $m=3$ at least).
	In the $\|\cdot\|_{0, A}$-norm, the error even follows the best approximation error for the weighted operator (bottom right), which is the best we can hope for.
	We moreover observe a larger pre-asymptotic range for $\varepsilon=2^{-4}$ than before for $\varepsilon=2^{-3}$ which indicates that the resolution condition between $k$, $\varepsilon$ and $H$ is sharp in this setting.
	
	\subsection{Periodic structure with local defects}
	We now place the periodic structure described in the previous section into a slab-like scatterer, i.e., we set
	\[\Omega_\varepsilon=\bigl((0.25, 0.75)\times (0,1))\bigr)\cap \bigcup_{j\in \gz^2}\varepsilon(j+(0.25, 0.75)^2).\]
	We fix the periodicity to $\varepsilon=2^{-3}$, see Figure \ref{fig:lensing} for a representation of $A$.
	Since the contrast is rather moderate, the results using $I_H$ and $I_H^1$ are very similar and we only depict computations with $I_H^1$.
	We choose the same function $f$ as data term, this time at $x_0=(0.25, 0.5)$, i.e., closer to the periodic structure.
	Moreover, we change the wave number and study $k=28$.
	The reference solution is computed on a fine mesh with $h=2^{-9}$.
	We emphasize that this a regime where the wave comes into (geometric) resonance with the periodic structure and where homogenization arguments can no longer be applied, see also the discussion in \cite{DS17wavenum}.

	\begin{figure}
		\includegraphics[width=0.47\textwidth,trim=40mm 95mm 45mm 98mm, clip=true, keepaspectratio=false]{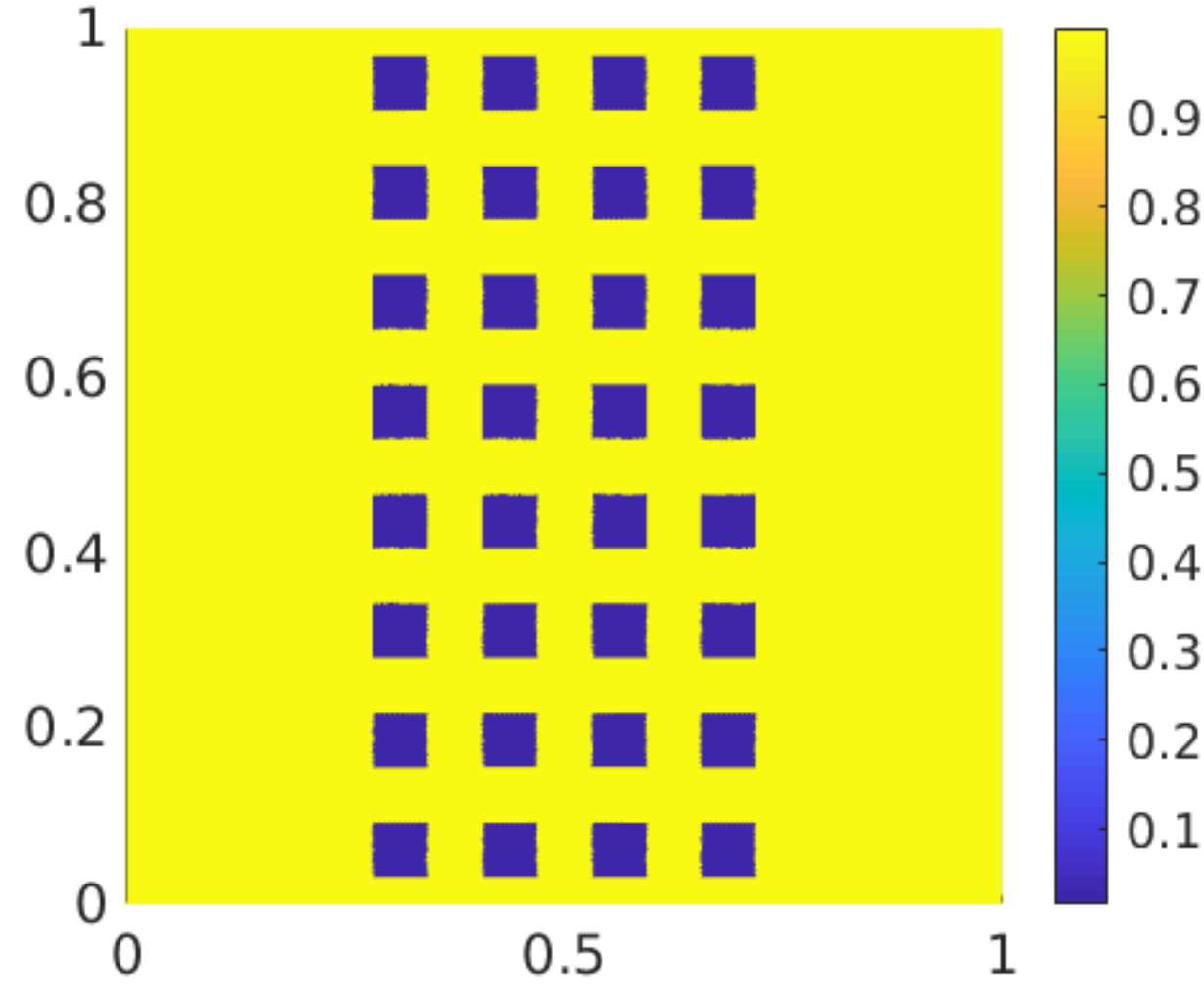}%
		\hspace{5ex}%
		\includegraphics[width=0.47\textwidth,trim=40mm 95mm 38mm 95mm, clip=true, keepaspectratio=false]{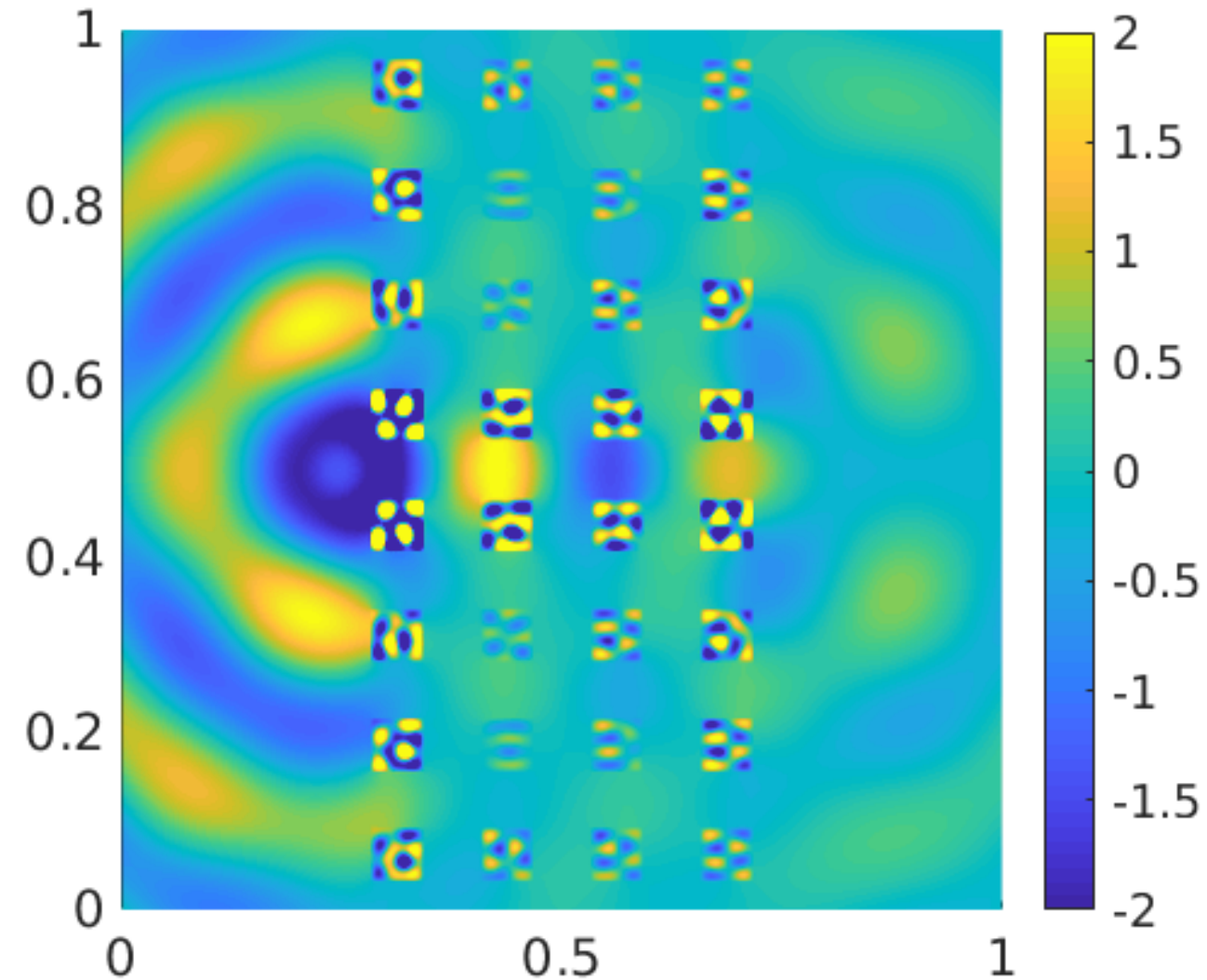}
		\caption{Lensing by a periodic structure, left: coefficient, right: (upscaled) LOD approximation $u_{\operatorname{LOD}, m}$ for $H=2^{-6}$, $m=2$.}	
		\label{fig:lensing}
	\end{figure}
	\begin{figure}
		\includegraphics[width=0.47\textwidth,trim=40mm 95mm 45mm 98mm, clip=true, keepaspectratio=false]{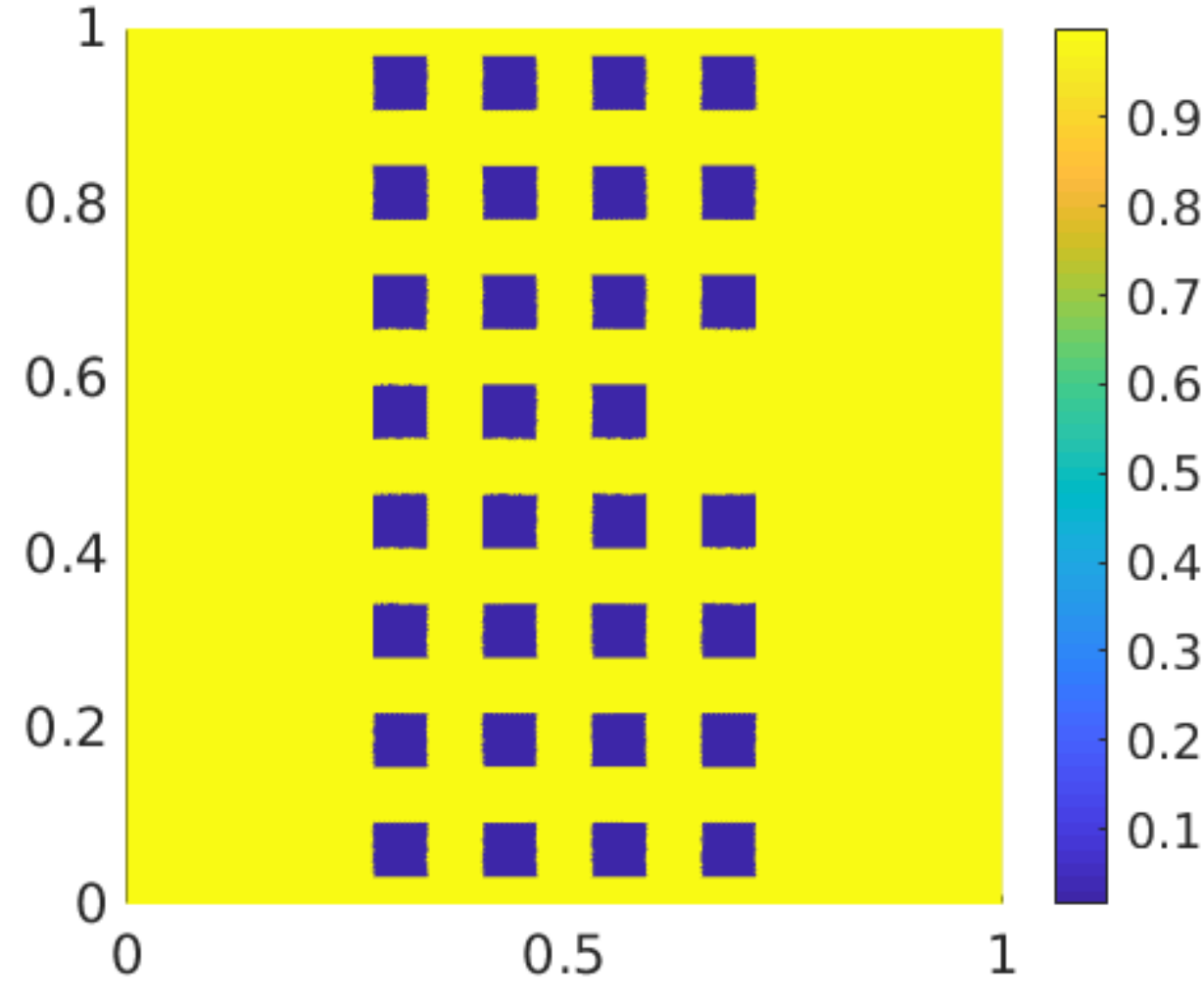}%
		\hspace{5ex}%
		\includegraphics[width=0.47\textwidth,trim=40mm 95mm 42mm 98mm, clip=true, keepaspectratio=false]{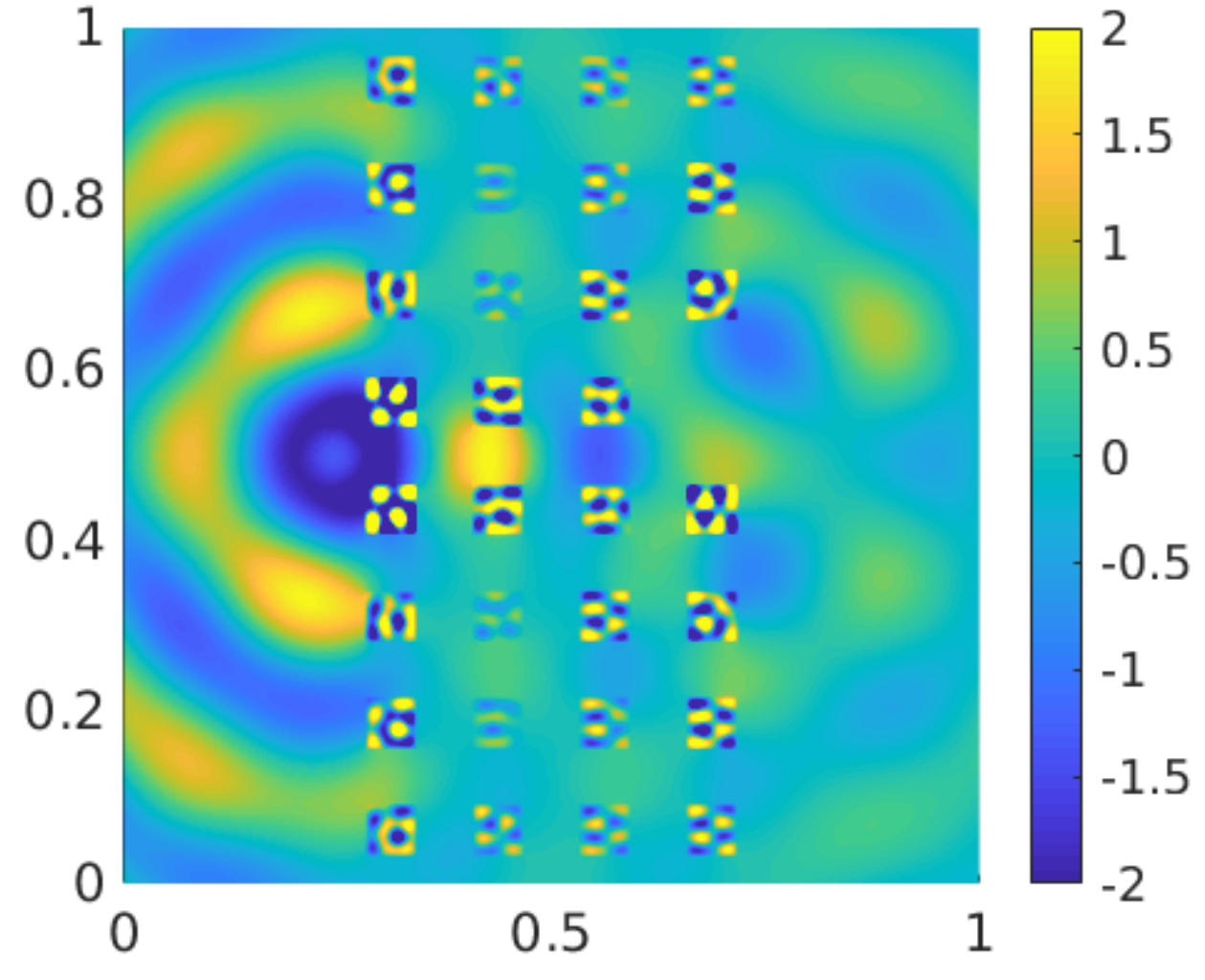}
		\caption{Periodic structure with a point defect, left: coefficient, right: (upscaled) LOD approximation $u_{\operatorname{LOD}, m}$ for $H=2^{-6}$, $m=2$.}	
		\label{fig:pointdef}
	\end{figure}

	In this section, we consider the periodic structure and two perturbations (a point and a line defect, see below), which yield physically interesting effects.
	As we mainly focus on the phenomena here, we do not study convergence histories, but show how (and for which meshes) the upscaled approximation $u_{\operatorname{LOD}, m}$ and its macroscopic part $u_{H, m}$ are able to faithfully represent the (qualitative) behavior of the exact solution.
	Compared to the wave number $k=28$, a rather coarse mesh of $H=2^{-6}$ and a moderate patch size of $m=2$ is sufficient to produce a qualitatively good approximation with our multiscale method in all experiments.
	Already in the previous subsection, we observed that the periodic setting with $\varepsilon=2^{-3}$ is not such a great challenge from the high contrast point of view so that a rough coupling of $kH\lesssim 1$ seems to be sufficient.
	We also observe that we need $H=2^{-7}$ and $m=2$ to have a faithful approximation by the macroscopic part $u_{H, m}$.
	All in all, the experiments of this section show that the multiscale method is able to capture physically relevant settings and is applicable also in the regime where no scale separation between wave length and fine-scale geometry exists.
	
	For the perfectly periodic setting, we observe a sort of lensing effect: The wave pattern behind the scatterer looks (qualitatively)  as if a source were located also behind the scatterer, see Figure \ref{fig:lensing}.
	This effect is closely related to negative refraction of the material, see \cite{DS17wavenum,CJJP02negrefraction}.
	Next, we introduce a point defect into the periodic structure: We eliminate one inclusion and leave the setting otherwise unchanged, see Figure \ref{fig:pointdef}.
	This local defect, however, although rather far away from the point source has a tremendous effect on the behavior of the wave. The wave pattern behind the scatterer is different from the periodic case (for instance, it is no longer symmetric around the axis $y=0.5$), see Figure \ref{fig:pointdef}. Hence, one may conclude that the lensing effect is destroyed.
	From the point defect we now move to a line defect: We widen the area where $A=1$ in the middle of the scatterer (around the line $y=0.5$) from $0.5\varepsilon$ to $\varepsilon$, see Figure \ref{fig:linedef}.
	The new (thin) channel now acts as a wave guide (cf. \cite{AS04waveguide}), i.e., the wave issued form the source is mainly traveling through the new channel, see Figure \ref{fig:linedef}.
	\begin{figure}
	\includegraphics[width=0.47\textwidth,trim=40mm 95mm 45mm 98mm, clip=true, keepaspectratio=false]{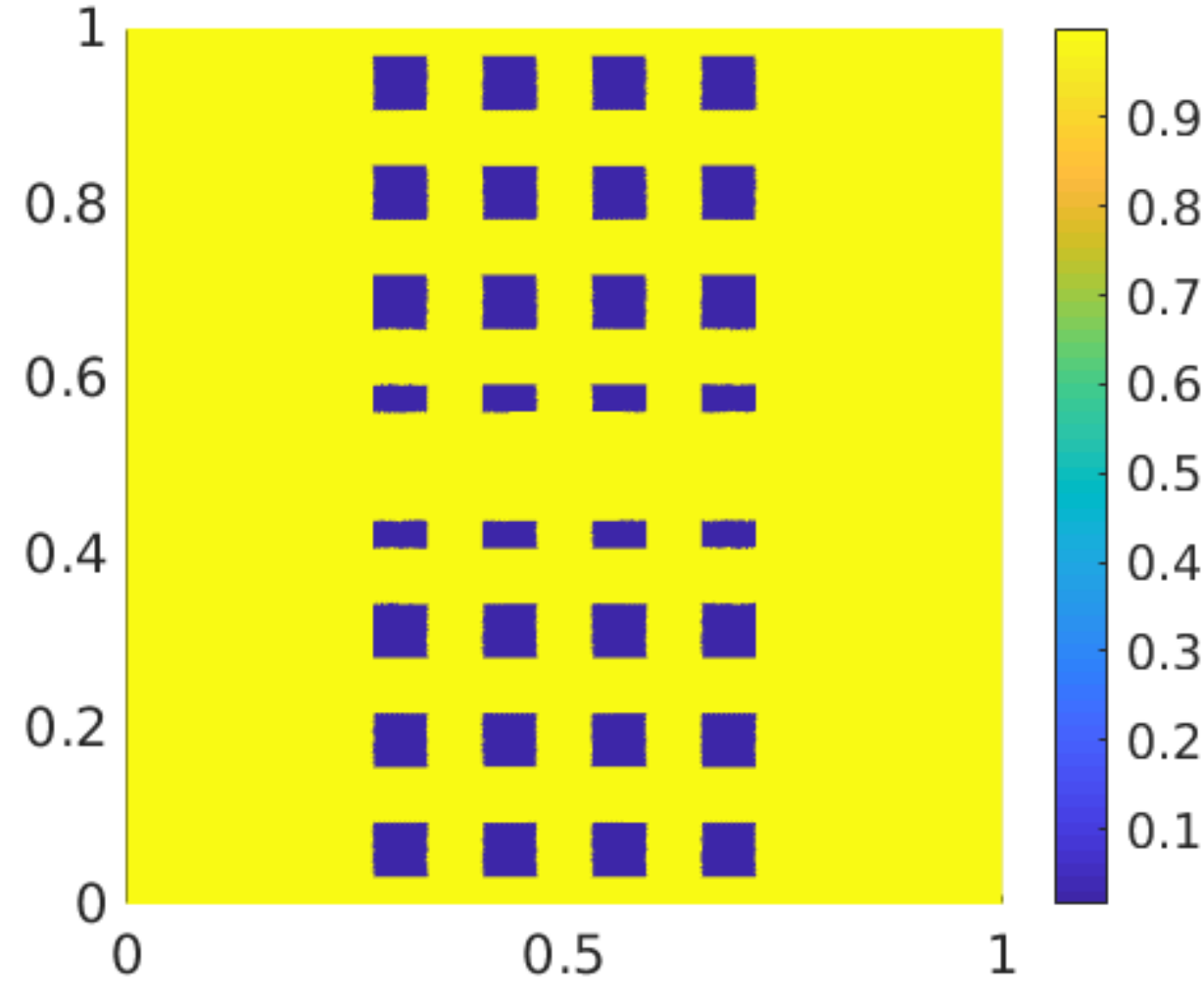}%
	\hspace{5ex}%
	\includegraphics[width=0.47\textwidth,trim=40mm 95mm 38mm 95mm, clip=true, keepaspectratio=false]{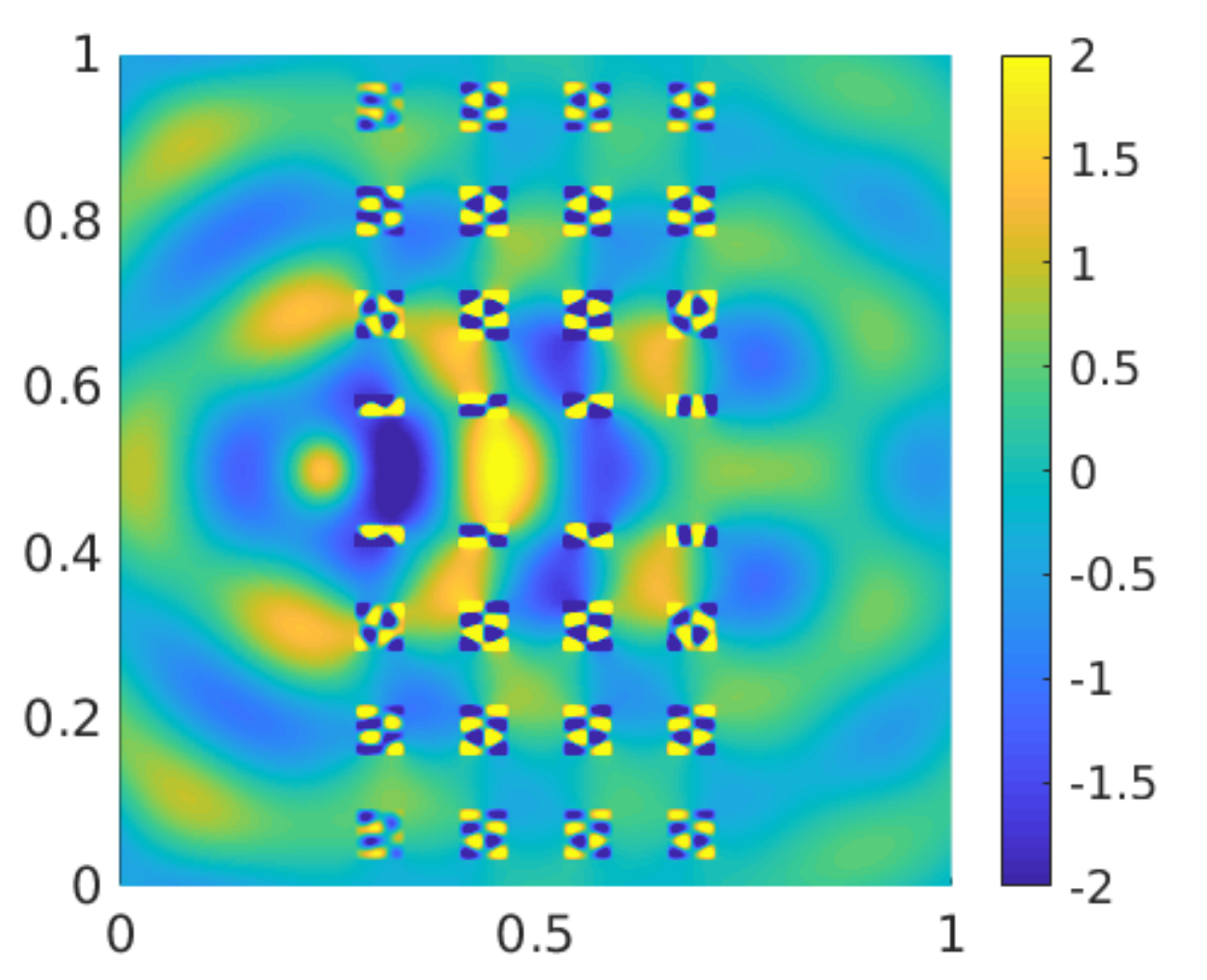}
	\caption{Wave guide induced by a line defect, left: coefficient, right: (upscaled) LOD approximation $u_{\operatorname{LOD}, m}$ for $H=2^{-6}$, $m=2$.}	
	\label{fig:linedef}
	\end{figure}

	\section*{Conclusion}
	We analyzed the Localized Orthogonal Decomposition for high-contrast high-frequency Helmholtz problems beyond periodicity and scale separation.
	It yields faithful approximations of the exact solution if the (coarse) mesh size is of the order of the effective wave length.
	The method relies on the solution of local finescale problems where the size of the local patches only grows logarithmically with the effective wave number.
	By effective wave number or wave length we denote the actual wave number or wave length divided by square root of the lowest value of the diffusion coefficient.
	We proved optimal convergence rates in the energy, the $L^2$, and a weighted $L^2$ norm and linked the results to expectations from the periodic case.
	The numerical experiments confirmed the theoretical convergence rates and showed the applicability of the method to the Bloch wave regime and periodic structures perturbed by local (point or line) defects.
	We thereby simulated interesting physical effects such as lensing and wave guides.
	Since the method is also applicable for totally unstructured or disordered coefficients, unusual and interesting phenomena in random media can be studied in future work.
	
	\section*{Acknowledgments}
	We thank S.~Sauter for his comment on the unique continuation principle for the Helmholtz equation and E.~A.~Spence for the discussion on stability estimates. We are grateful for the remarks of the anonymous reviewers which helped to improve the paper.

\end{document}